\documentclass[12pt]{amsart}
\usepackage{hyperref}
\usepackage{amsmath,amsthm,amssymb}
\usepackage{enumitem}
\usepackage{xcolor}
\usepackage[english]{babel}
\usepackage[utf8]{inputenc}
\usepackage[T1]{fontenc}
\usepackage{lmodern}
\usepackage{microtype}
\usepackage{mathrsfs}
\usepackage{booktabs}
\usepackage{pgfplots}
\pgfplotsset{compat=1.16}
\newcommand{\doi}[1]{\textsc{doi:}\,\href{https://doi.org/#1}{\texttt{#1}}}

\begin{document}
\title[Moduli of smoothness and the Laguerre transform]{Moduli of smoothness and growth of the Laguerre transform: two-sided estimates}
\author[N. Kakharman]{Nurbek Kakharman}
\address{
 Nurbek Kakharman:
 \endgraf
 SDU University
 \endgraf
 Kaskelen, Kazakhstan
 \endgraf
 and \endgraf
 Institute of Mathematics and Mathematical Modeling
 \endgraf
 28 Shevchenko str., 050010 Almaty, Kazakhstan
 \endgraf
 {\it E-mail address:} {\rm n.kakharman@math.kz; nurbek.kakharman@sdu.edu.kz}
 }
\author[N. Tokmagambetov]{Niyaz Tokmagambetov}
\address{
 Niyaz Tokmagambetov:
 \endgraf
 Institute of Mathematics and Mathematical Modeling
 \endgraf
 28 Shevchenko str., 050010 Almaty, Kazakhstan
 \endgraf
 {\it E-mail address:} {\rm niyaz.tokmagambetov@gmail.com; tokmagambetov@math.kz}
 }
 
\subjclass[2020]{Primary 42C10; Secondary 41A17, 41A25, 33C45, 44A15} 

\keywords{Laguerre transform; modulus of smoothness; Riemann-Lebesgue lemma; Hausdorff-Young inequality; Titchmarsh's theorem; best approximation; Jackson and Bernstein inequalities; saturation.}

\begin{abstract}
We study weighted coefficient estimates for the Laguerre transform in terms of moduli of smoothness generated by the Laguerre translation. For every integer $r\ge1$ and $1\le p\le2$, a weighted $\ell^{p'}$ norm of $\min\{1,nt\}^{r}\widehat f_\alpha(n)$ is bounded by the modulus of order $r$; for $2\le p\le\infty$ the reverse direction holds on the same weighted sequence scale. At $p=2$ these estimates yield a two-sided equivalence, and testing individual Laguerre polynomials proves optimality of the multiplier weight up to constants. The proof combines Hausdorff--Young inequalities with an averaged lower bound for the translation multiplier. As consequences we recover Jackson- and Bernstein-type estimates and the known $L^2$ approximation characterisation of the Lipschitz classes for $0<\beta<r$. We also give the corresponding one-sided coefficient-tail criteria for general $p$, together with saturation and endpoint statements.
\end{abstract}

\maketitle \numberwithin{equation}{section}
\newtheorem{theorem}{Theorem}[section]
 \newtheorem{proposition}[theorem]{Proposition}
\newtheorem{corollary}[theorem]{Corollary}
\newtheorem{lemma}[theorem]{Lemma}
\newtheorem{remark}[theorem]{Remark}
\newtheorem{problem}[theorem]{Problem}
\newtheorem{property}[theorem]{Property}
\newtheorem{definition}[theorem]{Definition}
\newtheorem{example}[theorem]{Example}
\allowdisplaybreaks

\section{Introduction}
The relation between small-step smoothness and the decay of Fourier coefficients or Fourier transforms is a classical theme of harmonic analysis; see Stein \cite{Stein}, Titchmarsh \cite{titchmarsh}, and Zygmund \cite{Zygmund}. For example, if $f\in L^1(-\pi,\pi)$ is extended periodically and
\[
\begin{aligned}
\widehat f(n)&=\frac1{2\pi}\int_{-\pi}^{\pi}f(x)e^{-inx}\,dx,\\
\omega(f;t)_1&=\sup_{|h|\le t}\frac1{4\pi}\int_{-\pi}^{\pi}|f(x+h)-f(x)|\,dx,
\end{aligned}
\]
then, for $n\ne0$,
\[
|\widehat f(n)|\le\frac1{4\pi}\int_{-\pi}^{\pi}\left|f\left(x+\frac\pi n\right)-f(x)\right|\,dx
\le\omega\left(f;\frac\pi{|n|}\right)_1.
\]
This estimate implies the Riemann--Lebesgue lemma and quantifies coefficient decay by a modulus of continuity.

Bray and Pinsky \cite{BP,BP2} developed weighted growth estimates for Fourier transforms using translation and spherical means. The following one-dimensional formulation is a model for our results.
\begin{proposition}[Fourier model]\label{pro1}
Let $\widehat f(\lambda)=\int_{\mathbb R}f(x)e^{-i\lambda x}\,dx$, with the usual Hausdorff--Young interpretation for $1<p\le2$, and put
\[
\Omega(f;t)_p=\sup_{0\le h\le t}\|f(\cdot+h)+f(\cdot-h)-2f\|_{L^p(\mathbb R)}.
\]
For every $t>0$ and $f\in L^p(\mathbb R)$,
\[
\begin{cases}
\displaystyle \sup_{\lambda\in\mathbb R}\min\{1,(\lambda t)^2\}|\widehat f(\lambda)|
\le C_1\Omega(f;t)_1,&p=1,\\[6pt]
\displaystyle \left(\int_{\mathbb R}\min\{1,(\lambda t)^2\}^{p'}|\widehat f(\lambda)|^{p'}\,d\lambda\right)^{1/p'}
\le C_p\Omega(f;t)_p,&1<p\le2,
\end{cases}
\]
where $p'=p/(p-1)$.
\end{proposition}
\begin{proof}
The symmetric difference has Fourier multiplier $-2(1-\cos(\lambda h))$. Moreover,
\[
\frac1t\int_0^t(1-\cos(\lambda h))\,dh
=1-\frac{\sin(\lambda t)}{\lambda t}
\ge c\min\{1,(\lambda t)^2\},
\]
with the quotient interpreted as $1$ at $\lambda=0$. The last inequality follows from Taylor's formula near zero, positivity on compact subsets of $\mathbb R\setminus\{0\}$, and $|\sin u/u|\le1/|u|$ for large $|u|$. Apply Hausdorff--Young to each symmetric difference. For $1<p\le2$, raise to the power $p'$, average over $h\in[0,t]$, and use Tonelli's theorem and Jensen's inequality. For $p=1$, use the $L^1$--$L^\infty$ estimate and average for each fixed $\lambda$.
\end{proof}
Higher-order versions and relations with weighted Fourier tails were studied by Dai and Ditzian \cite{DD}, Ditzian \cite{DZ,D2}, Gioev \cite{G}, Gorbachev and Tikhonov \cite{GTs}, and Bray \cite{Bray}. A fixed difference order has a finite saturation order, while higher powers of the difference operator describe higher smoothness. In the Hilbert-space setting, orthogonality also leads to two-sided coefficient estimates.

For the Laguerre setting, let ${\rm L}_n^{(\alpha)}$ denote the classical Laguerre polynomial, $\alpha>-1$. Its Rodrigues formula is
\[
{\rm L}_n^{(\alpha)}(x)=\frac{e^x x^{-\alpha}}{n!}\frac{d^n}{dx^n}\left(e^{-x}x^{n+\alpha}\right),
\]
and it satisfies
\begin{equation}\label{Leq}
\frac{d}{dx}\left(e^{-x}x^{\alpha+1}\frac{d}{dx}{\rm L}_n^{(\alpha)}(x)\right)
+ne^{-x}x^\alpha {\rm L}_n^{(\alpha)}(x)=0;
\end{equation}
see Szeg\H{o} \cite[Chapter~V]{Szego}. The classical Laguerre-transform convention of Debnath \cite{Deb1} uses
\begin{equation}\label{LT}
b_\alpha(f;n):=\int_0^\infty e^{-x}x^\alpha {\rm L}_n^{(\alpha)}(x)f(x)\,dx.
\end{equation}
Write
\begin{equation}\label{delta}
\lambda_n^\alpha:=\frac{\Gamma(n+\alpha+1)}{n!},
\qquad L_n^{(\alpha)}:=(\lambda_n^\alpha)^{-1/2}{\rm L}_n^{(\alpha)}.
\end{equation}
Throughout the paper, $\widehat f_\alpha(n)$ denotes the \emph{orthonormal} coefficient
\begin{equation}\label{LT-2}
L\{f\}(n)=\widehat f_\alpha(n):=\int_0^\infty e^{-x}x^\alpha L_n^{(\alpha)}(x)f(x)\,dx
=(\lambda_n^\alpha)^{-1/2}b_\alpha(f;n).
\end{equation}
For $f\in\mathcal L_\alpha^2=L^2((0,\infty),e^{-x}x^\alpha dx)$, orthogonal expansion gives
\begin{equation}\label{inlt-1}
f=\sum_{n=0}^\infty\widehat f_\alpha(n)L_n^{(\alpha)}
=\sum_{n=0}^\infty(\lambda_n^\alpha)^{-1}b_\alpha(f;n){\rm L}_n^{(\alpha)},
\end{equation}
with convergence in $\mathcal L_\alpha^2$. When $\alpha=0$, $\lambda_n^0=1$, and this agrees with the convention of McCully \cite{McC}. No pointwise inversion or inversion for arbitrary functions with existing coefficients is implicit in \eqref{inlt-1}; other norm-convergence statements are proved below.

Laguerre approximation has a substantial earlier literature. G\"orlich \cite{GM2} studied Ces\`aro and best approximation, and Abilov \cite{Abi} used the generalised Laguerre shift. Abilov, Abilova and Abilov \cite{AAA} employ arbitrary-order differences in the tensor-product setting and record the corresponding Lipschitz equivalence. Their derivative-free statement reduces to Theorem \ref{tm2} by taking $F(x,y)=f(x)$, second Laguerre parameter zero, difference order $k=r$, and $\nu=\beta/r$: only the zeroth mode in $y$ survives. They attribute that equivalence to Abilov \cite{Abi96} and Rzhavinskaya \cite{Rzh}. We recover this known $L^2$ result from the weighted estimates below.

Our emphasis is on a common weighted coefficient expression for the full range $1\le p\le\infty$ and on the averaged multiplier estimate that yields it. For $\alpha\ge0$, let $T_h^{(\alpha)}$ be the translation defined in Section \ref{sec2} and put
\[
\Omega_r(L;f;t)_p=\sup_{0\le h\le t}\|(I-T_h^{(\alpha)})^r f\|_{\mathcal L_\alpha^p},
\qquad r\in\mathbb N,\quad 0<t\le1.
\]
Here $\|f\|_{\mathcal L_\alpha^p}=\|e^{-x/2}f\|_{L^p(x^\alpha dx)}$, including the usual essential-supremum interpretation for $p=\infty$. The results are as follows.
\begin{itemize}
\item \emph{Weighted growth estimate} (Theorem \ref{estimates}). If $1\le p\le2$ and $f\in\mathcal L_\alpha^p$, then
\[
\left(\sum_{n\ge0}(\lambda_n^\alpha)^{1-p'/2}\min\{1,nt\}^{rp'}|\widehat f_\alpha(n)|^{p'}\right)^{1/p'}
\le C_{\alpha,p,r}\Omega_r(L;f;t)_p.
\]
For $p=1$, the left-hand side means
\[
\sup_{n\ge0}(\lambda_n^\alpha)^{-1/2}\min\{1,nt\}^r|\widehat f_\alpha(n)|.
\] The spectral parameter $n$ of the Laguerre operator gives the scale $nt$; in the Fourier model the corresponding scale is $(\lambda t)^2$.
\item \emph{Dual estimate} (Theorem \ref{estimatesdual}). If $2\le p\le\infty$ and $f\in\mathcal L_\alpha^p$, the same weighted expression bounds $\Omega_r(L;f;t)_p$ from above. Remark \ref{whyp2} rules out reversing the first estimate for any fixed $r$ when $p<2$.
\item \emph{$L^2$ equivalence and optimal order} (Theorem \ref{equiv} and Proposition \ref{sharp}). At $p=2$,
\[
\Omega_r(L;f;t)_2\asymp\left(\sum_{n\ge0}\min\{1,nt\}^{2r}|\widehat f_\alpha(n)|^2\right)^{1/2},
\]
with constants depending only on $\alpha,r$. Every nonnegative weight admissible in the corresponding lower estimate is bounded by a constant multiple of $\min\{1,nt\}^{r}$.
\item \emph{Approximation and Lipschitz classes} (Section \ref{sec5}). The $L^2$ estimate gives Jackson- and Bernstein-type inequalities and recovers
\[
f\in\mathrm{Lip}(\beta,r;\mathcal L_\alpha^2)
\iff E_m(f)_2=O(m^{-\beta}),\qquad 0<\beta<r.
\]
For general $p$, Theorem \ref{tm2p} instead uses the weighted coefficient tail $\mathcal E_m(f)_p$ defined in \eqref{Emp}: smoothness implies its decay for $p\le2$, and its decay implies smoothness for $p\ge2$. In the latter range, Proposition \ref{jbp} also gives $E_m(f)_p\lesssim\mathcal E_m(f)_p$. These tails equal best-approximation errors at $p=2$. The saturation and endpoint statements are given in Remarks \ref{sat} and \ref{betaeqr}.
\end{itemize}
Section \ref{sec2} develops the weighted sequence scale, Hausdorff--Young inequalities, and translation properties. Section \ref{sec3} proves the multiplier estimates, including an averaged lower bound. Sections \ref{sec4} and \ref{sec5} give the weighted growth inequalities and their approximation-theoretic consequences.

Unless a dependence is specified more precisely, constants may depend on the fixed parameters $\alpha,p,r$ and, in Lipschitz estimates, on $\beta$. They are independent of $f,n,m,h,t$. Constants implicit in an $O$-estimate for a fixed function may also depend on that function.
\section{Auxiliary results and notations}\label{sec2}
\subsection{Weighted Lebesgue measurable space}
For \(1 \leq p < \infty\) and \(\alpha > -1\), denote by $\mathcal{L}^p_{\alpha}(\mathbb{R}^+)$ the space of complex-valued Lebesgue measurable functions \(f\) on the half-line \( (0, \infty)\) and such that the function \(\left| f(x)e^{-x/2} \right|^p x^\alpha\) is integrable over \(\mathbb{R}^+\). The space $\mathcal{L}^p_{\alpha}$ is equipped with the norm
\begin{equation}
 \label{norm} \|f\|_{\mathcal{L}^p_{\alpha}} := \left( \int_{0}^{\infty} \left|e^{-x/2} f(x)\right|^p x^\alpha \, dx \right)^{\frac{1}{p}},
\end{equation}
and in the case $p=\infty$, we assume that $\mathcal{L}^{\infty}_{\alpha}$ is the space of measurable functions \(f\) such that the product $e^{-x/2} f(x)$ is essentially bounded on \(\mathbb{R}^+\),
\begin{equation*}
 \|f\|_{\mathcal{L}^{\infty}_\alpha} := \text{ess sup}\left\{|e^{-x/2} f(x)| : x \in (0, \infty) \right\}.
\end{equation*}
Equivalently, $f\mapsto e^{-x/2}f$ is an isometric isomorphism of $\mathcal{L}^p_{\alpha}$ onto $L^p\big((0,\infty),x^\alpha dx\big)$; we shall use this identification repeatedly.
The space $\mathcal{L}^2_{\alpha} $ is a Hilbert space with the inner product
\begin{equation}
 \label{inner} \langle f,g \rangle_{\mathcal{L}^2_{\alpha}} := \int_{0}^{\infty} \left(e^{-x/2} f(x)\right)\overline{\left(e^{-x/2} g(x)\right)} x^\alpha \, dx, \quad f,g \in \mathcal{L}^2_{\alpha}.
\end{equation}
To this end, we define the weighted ${l}^p_{\alpha}$ space of the sequences $\widehat{f}_{\alpha}(n)$ by the norm
\begin{equation}\label{norms}
\|\widehat{f}_{\alpha}\|_{{l}^p_{\alpha}}:=\left(\sum_{n=0}^{\infty}\left(\lambda_{n}^{\alpha}\right)^{\frac{2-p}{2}}\big|\widehat{f}_{\alpha}(n)\big|^p\right)^{\frac{1}{p}},~~~1\leq p<\infty,
 \end{equation}
 and in the case $p=\infty$, the corresponding norm is defined as
 \begin{equation}\label{norms-infty}
 \|\widehat{f}_{\alpha}\|_{{l}^{\infty}_{\alpha}}:=\sup_{n\geq0}\big| (\lambda_{n}^{\alpha})^{-1/2}\widehat{f}_{\alpha}(n)\big|,
 \end{equation}
where $\lambda^{\alpha}_n$ is the normalisation factor \eqref{delta}.
\begin{remark}\label{scale}
The exponent $\frac{2-p}{2}$ in \eqref{norms} gives a fixed-measure interpolation scale compatible with the $p=2$ and $p=\infty$ normalisations. Indeed, \eqref{norms} may be rewritten as
\[
\|\widehat f_\alpha\|_{{l}^p_\alpha}=\big\|(\lambda_n^\alpha)^{\frac1p-\frac12}\widehat f_\alpha(n)\big\|_{\ell^p},
\]
so that the weight attached to a single coefficient is $(\lambda_n^\alpha)^{\frac1p-\frac12}$: it equals $1$ for $p=2$, which is Plancherel's identity \eqref{pe} below, and tends to $(\lambda_n^\alpha)^{-1/2}$ as $p\to\infty$, which is \eqref{norms-infty}. Equivalently, ${l}^p_\alpha=L^p(\mathbb{N}_0,\nu_\alpha)$ for the \emph{fixed} measure $\nu_\alpha(\{n\})=\lambda_n^\alpha$, once the coefficients are normalised as $(\lambda_n^\alpha)^{-1/2}\widehat f_\alpha(n)$. This last observation is what makes the Riesz--Thorin theorem directly applicable in Theorem \ref{hyl}.
\end{remark}
\subsection{Basic properties of Laguerre transform}
The system of normalised Laguerre polynomials $\{L_n^{(\alpha)}(x)\}_{n=0}^{\infty}$ forms an orthonormal basis in the space $\mathcal{L}^2_{\alpha},$ and satisfies the following orthogonality condition \cite[p.100, Formula (5.1.1)]{Szego}
\begin{equation}\label{ort}\int_{0}^{\infty} e^{-x} x^{\alpha} L_n^{(\alpha)}(x) L_{m}^{(\alpha)}(x) \, dx = \lambda^{\alpha}_n (\lambda^{\alpha}_n)^{-1/2}(\lambda^{\alpha}_m)^{-1/2}\delta_{mn} = \delta_{nm}, \end{equation}
where $\delta_{mn}$ is the Kronecker delta symbol, and $\lambda^{\alpha}_n$ is the normalisation factor \eqref{delta}.
Furthermore, the norm of a function $f\in\mathcal{L}^2_{\alpha}$ can be expressed in terms of its Fourier--Laguerre coefficients by Plancherel identity \cite[p.43, (1.8)]{AV}:
\begin{equation}
 \label{pe} \|f\|_{\mathcal{L}^2_{\alpha}}^2=\sum_{n=0}^{\infty} |\widehat{f}_{\alpha}(n)|^2=\|\widehat{f}_{\alpha}\|_{{l}^2_{\alpha}}^2.
\end{equation}
We shall repeatedly use the following density statement.
\begin{lemma}\label{dense}
Let $\alpha>-1$ and $1\le p<\infty$. Then the linear span $\mathcal{P}_\alpha$ of $\{L_n^{(\alpha)}\}_{n\ge0}$ is dense in $\mathcal{L}^p_{\alpha}(\mathbb{R}^+)$.
\end{lemma}
\begin{proof}
Since $\mathcal{P}_\alpha$ coincides with the set of all algebraic polynomials, by the Hahn--Banach theorem it suffices to show that if a function $h\in L^{p'}\big((0,\infty),x^\alpha dx\big)$ satisfies
\begin{equation}\label{annih}
\int_{0}^{\infty}e^{-x/2}x^{k}h(x)\,x^{\alpha}dx=0\qquad\text{for all }k\in\mathbb{N}_0,
\end{equation}
then $h=0$ almost everywhere. Put $\Phi(z)=\int_{0}^{\infty}e^{-zx}h(x)x^{\alpha}dx$. By H\"older's inequality, for $\Re z>0$,
\[
\int_{0}^{\infty}e^{-\Re z\,x}|h(x)|x^{\alpha}dx\le \|h\|_{L^{p'}(x^\alpha dx)}\left(\int_{0}^{\infty}e^{-p\Re z\,x}x^{\alpha}dx\right)^{1/p}<\infty ,
\]
Thus $\Phi$ is analytic in $\Re z>0$, and differentiation under the integral gives
\[
\Phi^{(k)}(1/2)=(-1)^k\int_0^\infty e^{-x/2}x^kh(x)x^\alpha\,dx.
\] By \eqref{annih} all derivatives of $\Phi$ vanish at $z=1/2$, hence $\Phi\equiv 0$ on $\Re z>0$. By the uniqueness theorem for the Laplace transform, $h(x)x^{\alpha}=0$ for almost every $x>0$, that is, $h=0$ a.e.
\end{proof}
\begin{lemma}[$\mathcal{L}^1_{\alpha}$--${l}^{\infty}_{\alpha}$ estimate]\label{lem1}
Let $f\in\mathcal{L}^1_{\alpha},$ $\alpha\geq 0$, such that
$$
\|f\|_{\mathcal{L}^1_{\alpha}}=\int_{0}^{\infty} \left|e^{-x/2} f(x)\right| x^\alpha \, dx<\infty.
$$
 Then
\begin{equation}\label{L1linfty}
\|\widehat{f}_{\alpha}\|_{{l}^{\infty}_{\alpha}}=\sup_{n\geq0}\big| (\lambda_{n}^{\alpha})^{-1/2} \widehat{f}_{\alpha}(n)\big|\leq \frac{1}{\Gamma(\alpha+1)}\,\|f\|_{\mathcal{L}^1_{\alpha}}.
 \end{equation}
\end{lemma}
\begin{proof} Let us first recall some useful estimates. Watson \cite[p. 21]{watson}, obtained the following estimate for Laguerre polynomials for \(\alpha \geq 0\):
\begin{equation}
 \label{westimate}
|e^{-x/2} {\rm L}^{(\alpha)}_n(x)| \leq \frac{\Gamma(n + \alpha + 1)}{n! \Gamma(\alpha + 1)}, \quad x \geq 0.\end{equation}
This estimate for \(\alpha = 0\) was obtained earlier by Szeg\H{o} by another method, and he proved that it is strict for \(x > 0\).
By \eqref{LT-2} and \eqref{westimate} we have
 \begin{align*}
 \big|\widehat{f}_{\alpha}(n)\big| &\leq \int_{0}^{\infty} \big|e^{-x} x^{\alpha} L_n^{(\alpha)}(x) f(x)\big| \, dx\\
 &\leq \int_{0}^{\infty} \big|e^{-x/2} f(x)\big|x^{\alpha} \big|e^{-x/2} L_n^{(\alpha)}(x)\big| \, dx\\
 &\leq\frac{\Gamma(n + \alpha + 1)}{n! \Gamma(\alpha + 1)} (\lambda_{n}^{\alpha})^{-1/2} \int_{0}^{\infty}\big|e^{-x/2} f(x)\big|x^{\alpha} \, dx\\
 &= \frac{1}{\Gamma(\alpha + 1)} (\lambda_{n}^{\alpha})^{1/2} \|f\|_{\mathcal{L}^1_{\alpha}},
 \end{align*}
 which is \eqref{L1linfty}.
\end{proof}
\begin{corollary}[Riemann--Lebesgue lemma]\label{RL}
Let $\alpha\ge0$ and $f\in\mathcal{L}^1_{\alpha}$. Then
\[
(\lambda_n^{\alpha})^{-1/2}\widehat{f}_{\alpha}(n)=o(1),\qquad n\to\infty.
\]
\end{corollary}
\begin{proof}
Let $\varepsilon>0$. By Lemma \ref{dense} there is $P=\sum_{k=0}^{N}a_kL_k^{(\alpha)}\in\mathcal{P}_\alpha$ with $\|f-P\|_{\mathcal{L}^1_{\alpha}}<\varepsilon\,\Gamma(\alpha+1)$. By \eqref{ort}, $\widehat{P}_\alpha(n)=0$ for $n>N$, whence by \eqref{L1linfty}
\begin{align*}
(\lambda_n^{\alpha})^{-1/2}|\widehat f_\alpha(n)|
&=(\lambda_n^{\alpha})^{-1/2}|\widehat{(f-P)}_\alpha(n)|\\
&\le\frac{\|f-P\|_{\mathcal L_\alpha^1}}{\Gamma(\alpha+1)}<\varepsilon,\qquad n>N.\qedhere
\end{align*}
\end{proof}
\begin{theorem}[Hausdorff-Young inequality]
 \label{hyl} Let $f\in\mathcal{L}^p_{\alpha},$ $\alpha\geq 0$, $p\in[1,2]$, and let $p'$ be the conjugate exponent of $p$. Then
 \begin{equation}
 \label{hyineq} \|\widehat{f}_{\alpha}\|_{{l}^{p'}_{\alpha}}\leq c_{\alpha,p}\|f\|_{\mathcal{L}^p_{\alpha}},\qquad c_{\alpha,p}:=\Gamma(\alpha+1)^{-\left(\frac{2}{p}-1\right)}.
 \end{equation}
 Explicitly,
 \[
 \left(\sum_{n=0}^{\infty}(\lambda^{\alpha}_n)^{\frac{2-p'}{2}}\big|\widehat{f}_{\alpha}(n)\big|^{p'}\right)^{1/p'}\le c_{\alpha,p}\|f\|_{\mathcal{L}^p_{\alpha}},\qquad 1<p\le 2.
 \]
\end{theorem}
\begin{proof}
Here we use the language of normed vector spaces and bounded linear maps, as is convenient for the application of the Riesz--Thorin theorem. As explained in Remark \ref{scale}, both endpoint spaces have to be realised over one and the same pair of measure spaces. Let
\[
d\mu(x)=x^{\alpha}dx \ \text{ on } (0,\infty),\qquad \nu_\alpha(\{n\})=\lambda^{\alpha}_n \ \text{ on } \mathbb{N}_0,
\]
put $g(x)=e^{-x/2}f(x)$, so that $\|f\|_{\mathcal{L}^p_{\alpha}}=\|g\|_{L^p(\mu)}$, and define
\[
Ug(n):=(\lambda^{\alpha}_n)^{-1/2}\widehat{f}_{\alpha}(n)=\int_{0}^{\infty}K_n(x)\,g(x)\,d\mu(x),
\]
where $K_n(x):=(\lambda^{\alpha}_n)^{-1/2}e^{-x/2}L_n^{(\alpha)}(x)$.
By \eqref{pe},
\[
\|Ug\|^2_{L^2(\nu_\alpha)}=\sum_{n\ge0}\lambda^{\alpha}_n(\lambda^{\alpha}_n)^{-1}\big|\widehat{f}_{\alpha}(n)\big|^2=\|f\|^2_{\mathcal{L}^2_{\alpha}}=\|g\|^2_{L^2(\mu)},
\]
so $U:L^2(\mu)\to L^2(\nu_\alpha)$ is an isometry, while by Lemma \ref{lem1}
\[
\|Ug\|_{L^\infty(\nu_\alpha)}=\sup_{n\ge0}(\lambda^{\alpha}_n)^{-1/2}\big|\widehat{f}_{\alpha}(n)\big|\le\frac{1}{\Gamma(\alpha+1)}\|g\|_{L^1(\mu)} ,
\]
that is, $U:L^1(\mu)\to L^\infty(\nu_\alpha)$ is bounded with norm at most $\Gamma(\alpha+1)^{-1}$. For \( 0 < \theta < 1 \) define the exponents \( p\) and \( q \) via
\[
\frac{1}{p} = \frac{1-\theta}{2} + \frac{\theta}{1}, \quad \frac{1}{q} = \frac{1-\theta}{2} + \frac{\theta}{\infty} = \frac{1-\theta}{2},
\]
so that \( p \in (1, 2) \), $\theta=\frac2p-1$ and $q=p'$. Since the two endpoint statements are formulated over the same pair of measure spaces $(\mu,\nu_\alpha)$, the Riesz--Thorin theorem \cite{Graf} applies and gives
\[
\|U\|_{L^{p}(\mu) \to L^{p'}(\nu_\alpha)} \leq \|U\|_{L^2(\mu) \to L^2(\nu_\alpha)}^{1-\theta} \|U\|_{L^1(\mu) \to L^{\infty}(\nu_\alpha)}^{\theta}\le \Gamma(\alpha+1)^{-\left(\frac2p-1\right)}.
\]
It remains to observe that
\begin{align*}
\|Ug\|^{p'}_{L^{p'}(\nu_\alpha)}
&=\sum_{n\ge0}\lambda_n^\alpha(\lambda_n^\alpha)^{-p'/2}|\widehat f_\alpha(n)|^{p'}\\
&=\sum_{n\ge0}(\lambda_n^\alpha)^{(2-p')/2}|\widehat f_\alpha(n)|^{p'}
=\|\widehat f_\alpha\|_{{l}^{p'}_\alpha}^{p'}.
\end{align*}
The cases $p=1$ and $p=2$ are Lemma \ref{lem1} and \eqref{pe}, respectively.
\end{proof}
The next two statements are the duals of Lemma \ref{lem1} and Theorem \ref{hyl}; they will give the counterpart of our main estimate in the range $p\ge2$.
\begin{lemma}[${l}^{1}_{\alpha}$--$\mathcal{L}^{\infty}_{\alpha}$ estimate]\label{lem1dual}
Let $\alpha\ge0$ and let $\{a_n\}_{n\ge0}$ be a sequence of complex numbers with $\sum_{n\ge0}(\lambda^{\alpha}_n)^{1/2}|a_n|<\infty$. Then the series $g:=\sum_{n\ge0}a_nL^{(\alpha)}_n$ converges in the sense that $\sum_{n\ge0}a_n\,e^{-x/2}L^{(\alpha)}_n(x)$ converges absolutely and uniformly on $[0,\infty)$; moreover $g\in\mathcal{L}^{\infty}_{\alpha}$, $\widehat{g}_{\alpha}(n)=a_n$ for every $n$, and
\begin{equation}\label{l1Linfty}
\|g\|_{\mathcal{L}^{\infty}_{\alpha}}\ \le\ \frac{1}{\Gamma(\alpha+1)}\,\|\widehat{g}_{\alpha}\|_{{l}^{1}_{\alpha}}.
\end{equation}
\end{lemma}
\begin{proof}
By \eqref{westimate} and $L^{(\alpha)}_n=(\lambda^{\alpha}_n)^{-1/2}{\rm L}^{(\alpha)}_n$, for every $x\ge0$
\begin{equation}\label{unifbound}
e^{-x/2}\big|L^{(\alpha)}_n(x)\big|\ \le\ (\lambda^{\alpha}_n)^{-1/2}\,{\rm L}^{(\alpha)}_n(0)\ =\ \frac{(\lambda^{\alpha}_n)^{1/2}}{\Gamma(\alpha+1)}.
\end{equation}
Hence $\sum_{n}|a_n|\,e^{-x/2}|L^{(\alpha)}_n(x)|\le\Gamma(\alpha+1)^{-1}\sum_{n}(\lambda^{\alpha}_n)^{1/2}|a_n|<\infty$ uniformly in $x\ge0$, which gives both the uniform convergence and \eqref{l1Linfty}, the last norm being $\sum_{n}(\lambda^{\alpha}_n)^{1/2}|a_n|$ by \eqref{norms} with $p=1$. Finally, $e^{-x/2}g$ is bounded and $e^{-x/2}L^{(\alpha)}_m\in L^1(x^{\alpha}dx)$, so the series may be integrated term by term against $e^{-x/2}L^{(\alpha)}_m(x)x^{\alpha}$, and \eqref{ort} gives $\widehat{g}_{\alpha}(m)=a_m$.
\end{proof}
\begin{theorem}[Dual Hausdorff--Young inequality]\label{hyldual}
 Let $\alpha\geq 0$ and $2\le p\le\infty$, and let $p'\in[1,2]$ be the conjugate exponent of $p$. Put $c^{*}_{\alpha,p}:=\Gamma(\alpha+1)^{-\left(1-\frac{2}{p}\right)}$.
\begin{itemize}
\item[\rm(i)] If $\{a_n\}_{n\ge0}$ satisfies $\|a\|_{{l}^{p'}_{\alpha}}:=\big(\sum_{n}(\lambda^{\alpha}_n)^{\frac{2-p'}{2}}|a_n|^{p'}\big)^{1/p'}<\infty$, then the series $\sum_{n}a_nL^{(\alpha)}_n$ converges in $\mathcal{L}^p_{\alpha}$ (uniformly after multiplication by $e^{-x/2}$ when $p=\infty$) to a function $g$ with $\widehat{g}_{\alpha}(n)=a_n$ for all $n$ and
\begin{equation}\label{hyineqdual}
\|g\|_{\mathcal{L}^p_{\alpha}}\ \leq\ c^{*}_{\alpha,p}\,\|a\|_{{l}^{p'}_{\alpha}}.
\end{equation}
\item[\rm(ii)] Consequently, if $f\in\mathcal{L}^p_{\alpha}$ then
\begin{equation}\label{hyineqdual2}
\|f\|_{\mathcal{L}^p_{\alpha}}\ \leq\ c^{*}_{\alpha,p}\,\|\widehat{f}_{\alpha}\|_{{l}^{p'}_{\alpha}}.
\end{equation}
\end{itemize}
\end{theorem}
\begin{remark}\label{needf}
The hypothesis $f\in\mathcal{L}^p_{\alpha}$ in {\rm(ii)} cannot be dropped in favour of the finiteness of $\|\widehat{f}_{\alpha}\|_{{l}^{p'}_{\alpha}}$ alone. Indeed, by the classical example of Stieltjes, $\int_{0}^{\infty}x^{k}e^{-x^{1/4}}\sin\big(x^{1/4}\big)dx=0$ for every $k\in\mathbb{N}_0$, so the function $f(x)=e^{x}x^{-\alpha}e^{-x^{1/4}}\sin\big(x^{1/4}\big)$ has an absolutely convergent Laguerre transform \eqref{LT-2} with $\widehat{f}_{\alpha}(n)=0$ for every $n$, while $e^{-x/2}f$ is unbounded and belongs to no $L^p(x^{\alpha}dx)$.
\end{remark}
\begin{proof}
{\rm(i)} We keep the measures $d\mu(x)=x^{\alpha}dx$ and $\nu_\alpha(\{n\})=\lambda^{\alpha}_n$ of the proof of Theorem \ref{hyl} and consider the map
\[
Va:=\sum_{n\ge0}(\lambda^{\alpha}_n)^{1/2}a_n\,L^{(\alpha)}_n.
\]
By the computation in the proof of Theorem \ref{hyl}, for every $q\in[1,\infty]$,
\[
\|a\|_{L^{q}(\nu_\alpha)}=\Big(\sum_{n\ge0}(\lambda^{\alpha}_n)^{\frac{2-q}{2}}\big|(\lambda^{\alpha}_n)^{1/2}a_n\big|^{q}\Big)^{1/q}.
\] By \eqref{pe}, $\|Va\|^2_{\mathcal{L}^2_{\alpha}}=\sum_n\lambda^{\alpha}_n|a_n|^2=\|a\|^2_{L^2(\nu_\alpha)}$, so $V:L^2(\nu_\alpha)\to\mathcal{L}^2_{\alpha}$ is an isometry; and by Lemma \ref{lem1dual} applied to the coefficients $b_n:=(\lambda^{\alpha}_n)^{1/2}a_n$, whose hypothesis $\sum_n(\lambda^{\alpha}_n)^{1/2}|b_n|=\sum_n\lambda^{\alpha}_n|a_n|<\infty$ is exactly $a\in L^1(\nu_\alpha)$,
\[
\|Va\|_{\mathcal{L}^{\infty}_{\alpha}}\le\frac{1}{\Gamma(\alpha+1)}\sum_{n\ge0}\lambda^{\alpha}_n|a_n|=\frac{1}{\Gamma(\alpha+1)}\|a\|_{L^{1}(\nu_\alpha)} ,
\]
so $V:L^{1}(\nu_\alpha)\to\mathcal{L}^{\infty}_{\alpha}$ is bounded with norm at most $\Gamma(\alpha+1)^{-1}$; in both cases $\widehat{(Va)}_{\alpha}(n)=(\lambda^{\alpha}_n)^{1/2}a_n$, by \eqref{ort} and Lemma \ref{lem1dual} respectively. For $0<\theta<1$ define $p,p'$ by
\[
\frac1p=\frac{1-\theta}{2},\qquad \frac{1}{p'}=\frac{1-\theta}{2}+\frac{\theta}{1},
\]
so that $p\in(2,\infty)$, $\theta=1-\frac2p$ and $p'$ is conjugate to $p$. Apply Riesz--Thorin to the conjugated operator $a\mapsto e^{-x/2}Va$, whose endpoint target spaces are $L^2(\mu)$ and $L^\infty(\mu)$ over the same measure $d\mu=x^\alpha dx$. The Riesz--Thorin theorem \cite{Graf} therefore applies and gives $\|V\|_{L^{p'}(\nu_\alpha)\to\mathcal{L}^p_{\alpha}}\le\Gamma(\alpha+1)^{-\theta}$. Applying this to the sequence $\{(\lambda^{\alpha}_n)^{-1/2}a_n\}$, whose $L^{p'}(\nu_\alpha)$ norm is $\|a\|_{{l}^{p'}_{\alpha}}$, gives \eqref{hyineqdual}. The cases $p=2$ and $p=\infty$ are \eqref{pe} and Lemma \ref{lem1dual}.
{\rm(ii)} We may assume $\|\widehat{f}_{\alpha}\|_{{l}^{p'}_{\alpha}}<\infty$. Let $g$ be the function produced by {\rm(i)} for $a_n=\widehat{f}_{\alpha}(n)$, so that $g\in\mathcal{L}^p_{\alpha}$ and $\widehat{g}_{\alpha}=\widehat{f}_{\alpha}$. Then $w:=e^{-x/2}(f-g)\in L^{p}(\mu)$ satisfies $\int_{0}^{\infty}w(x)\,e^{-x/2}L^{(\alpha)}_m(x)\,x^{\alpha}dx=0$ for every $m$. Since $2\le p\le\infty$ we have $1\le p'<\infty$, so by Lemma \ref{dense} the functions $e^{-x/2}L^{(\alpha)}_m$ span a dense subspace of $L^{p'}(\mu)$, and $w$ annihilates it; hence $w=0$ a.e.\ and $f=g$ a.e. Now \eqref{hyineqdual2} is \eqref{hyineqdual}.
\end{proof}
\subsection{Laguerre translation operator}
In \cite[p. 21]{watson}, a product formula for Laguerre polynomials was derived: for $\alpha>-\frac12$,
\begin{equation}\begin{split}\label{pf}
 \frac{n! {\rm L}_n^{(\alpha)}(x) {\rm L}_n^{(\alpha)}(y) }{\Gamma(n+\alpha+1)}=\frac{1}{\sqrt{\pi}}\int\limits_{0}^{\pi}\exp{\left(-\sqrt{xy}\cos{\theta}\right)}\frac{J_{\alpha-\frac{1}{2}}\left(\sqrt{xy}\sin{\theta}\right)}{\left(\frac{1}{2}\sqrt{xy}\sin{\theta}\right)^{\alpha-\frac{1}{2}}}\sin^{2\alpha}{\theta}&\\
 \times {\rm L}_n^{(\alpha)}\left(x+y+2\sqrt{xy}\cos{\theta}\right)d\theta&,
 \end{split}\end{equation}
 where $J_{\nu}(\cdot)$ is the Bessel function of the first kind of order $\nu.$ The product formulas have several applications.
 \textbf{i)} In particular, using the product formula \eqref{pf}, for $ \alpha\geq0,$ the estimate \eqref{westimate} for Laguerre polynomials was derived. Using the identity ${\rm L}_n^{(\alpha)}(0) = \binom{n+\alpha}{n}=\frac{\Gamma(n+\alpha+1)}{n!\Gamma(\alpha+1)},$
as presented in Szeg\H{o}'s seminal work \cite[p.101, (5.1.7)]{Szego}, the estimate \eqref{westimate} can be reformulated in the compact form
\begin{equation}\label{est1}
 e^{-x/2}\big|{\rm L}_n^{(\alpha)}(x)\big|\leq {\rm L}_n^{(\alpha)}(0),~~~~x\geq 0,\ \alpha\ge0.
\end{equation}
Equivalently, the results can be written in terms of the normalised Laguerre functions
\begin{equation*}
 R_{n}^{(\alpha)}(h):=\frac{{\rm L}_{n}^{(\alpha)}(h)}{{\rm L}_{n}^{(\alpha)}(0)},~~~h\geq0,
\end{equation*}
which are orthogonal on $[0,\infty)$ \cite{GM1} (see also \cite{GM}), and by \eqref{est1} we have
\begin{equation}\label{est2}
 e^{-h/2}|R_{n}^{(\alpha)}(h)|\leq1,\qquad\text{equivalently}\qquad |R_{n}^{(\alpha)}(h)|\le e^{h/2},~~~h\geq0.
\end{equation}
\begin{remark} The normalised Laguerre function equals the confluent hypergeometric function \cite{Szego}, or Kummer's function of the first kind, $R_{n}^{(\alpha)}(z)={}_1F_1(-n;\alpha+1;z)=M(-n;\alpha+1;z).$
 However, for consistency, we adopt the notation from \cite{GM1}.
\end{remark}
\textbf{ii) } The product formula \eqref{pf} provides the basis for the definition of the generalised translation operator associated with the Laguerre translation for $ \alpha>-\frac{1}{2}.$ Let us denote the Laguerre translation operator by $T_{h}^{(\alpha)}(f;x)$ for $x,h>0$ given as follows
\begin{equation}\begin{split}\label{tr-op}
 T_{h}^{(\alpha)}(f;x)=\frac{\Gamma(\alpha+1)2^{\alpha}}{\sqrt{2\pi}}\int\limits_{0}^{\pi}f\left(x+h+2\sqrt{xh}\cos{\theta}\right)
 \exp{\left(-\sqrt{xh}\cos{\theta}\right)}&\\
 \times\frac{J_{\alpha-\frac{1}{2}}\left(\sqrt{xh}\sin{\theta}\right)}{\left(\sqrt{xh}\sin{\theta}\right)^{\alpha-\frac{1}{2}}}\sin^{2\alpha}{\theta} d\theta&,
 \end{split}\end{equation}
with step $h\in[0;\infty)$, and $T_{h}^{(\alpha)}(f;0)=f(h)$ for $h>0$; $T_{0}^{(\alpha)}(f;x)=f(x)$ for $x\geq0$, see \cite[p.162, formula (1.2)]{GM1} (see also \cite[Section~2]{AV}, \cite{GM}).
\begin{lemma}[Norm of the Laguerre translation]\label{Tnorm}
Let $\alpha\ge0$, $h\ge0$ and $1\le p\le\infty$. For every $f\in\mathcal{L}^p_{\alpha}(\mathbb{R}^+)$ the integral \eqref{tr-op} converges absolutely for almost every $x>0$, $T^{(\alpha)}_hf\in\mathcal{L}^p_{\alpha}$, and
\begin{equation}\label{Tnormbound}
\big\|T^{(\alpha)}_{h}f\big\|_{\mathcal{L}^p_{\alpha}}\ \le\ e^{h/2}\,\|f\|_{\mathcal{L}^p_{\alpha}} .
\end{equation}
For $p=\infty$ the constant $e^{h/2}$ cannot be improved. In particular the norms of $T^{(\alpha)}_{h}$ on $\mathcal{L}^p_{\alpha}$, $0\le h\le1$, are bounded by $e^{1/2}$ for every $1\le p\le\infty$.
\end{lemma}
\begin{proof}
For $h=0$ all assertions follow from $T_0^{(\alpha)}=I$, so assume $h>0$.
Write $j_\nu(t):=\Gamma(\nu+1)(2/t)^{\nu}J_\nu(t)$ for the normalised Bessel function, so that $j_\nu(0)=1$, and put
\[
c_\alpha:=\frac{\Gamma(\alpha+1)}{\sqrt\pi\,\Gamma(\alpha+\tfrac12)}=\Big(\int_0^\pi\sin^{2\alpha}\theta\,d\theta\Big)^{-1}.
\]
Since $\frac{\Gamma(\alpha+1)2^{\alpha}}{\sqrt{2\pi}}\cdot\frac{1}{\Gamma(\alpha+\frac12)2^{\alpha-1/2}}=c_\alpha$, formula \eqref{tr-op} reads
\[
T^{(\alpha)}_{h}(f;x)=c_\alpha\int_0^\pi f(z_\theta)\,e^{-\sqrt{xh}\cos\theta}\,j_{\alpha-\frac12}\big(\sqrt{xh}\sin\theta\big)\sin^{2\alpha}\theta\,d\theta,
\]
where $z_\theta:=x+h+2\sqrt{xh}\cos\theta$.
By Poisson's integral,
\[
j_\nu(t)=\frac{\Gamma(\nu+1)}{\sqrt\pi\,\Gamma(\nu+\frac12)}\int_0^\pi\cos(t\cos\varphi)\sin^{2\nu}\varphi\,d\varphi\qquad(\nu>-\tfrac12),
\]
while $j_{-1/2}(t)=\cos t$; hence $|j_{\alpha-1/2}|\le1$ on $\mathbb{R}$ for every $\alpha\ge0$. Put $F:=e^{-x/2}f$, so that $\|f\|_{\mathcal{L}^p_{\alpha}}=\|F\|_{L^p(x^\alpha dx)}$, and observe that $e^{-x/2}\,e^{z_\theta/2}\,e^{-\sqrt{xh}\cos\theta}=e^{h/2}$ identically. Therefore
\begin{equation}\label{Tpointwise}
e^{-x/2}\big|T^{(\alpha)}_{h}(f;x)\big|\ \le\ e^{h/2}\,c_\alpha\int_0^\pi|F(z_\theta)|\sin^{2\alpha}\theta\,d\theta\ =:\ e^{h/2}\,(M_h|F|)(x).
\end{equation}
The operator $M_h$ is an average with respect to the probability measure $c_\alpha\sin^{2\alpha}\theta\,d\theta$ on $(0,\pi)$. For $p=\infty$, \eqref{Tpointwise} gives \eqref{Tnormbound} at once. For $p<\infty$, apply Tonelli's theorem to $g=|F|^p$. The identity below shows that $M_h(|F|^p)$ is finite almost everywhere and has integral $\|F\|_{L^p(x^\alpha dx)}^p$. Jensen's inequality gives $(M_h|F|)^p\le M_h(|F|^p)$, proving both absolute convergence of \eqref{tr-op} almost everywhere and \eqref{Tnormbound}. The required identity is
\begin{equation}\label{masspres}
\int_0^\infty (M_hg)(x)\,x^{\alpha}\,dx=\int_0^\infty g(z)\,z^{\alpha}\,dz\qquad\text{for every measurable } g\ge0 .
\end{equation}
To prove \eqref{masspres} substitute $x=s^2$, $h=a^2$, $z_\theta=\zeta^2$. For fixed $s$, as $\theta$ runs over $(0,\pi)$ the variable $\zeta=\sqrt{s^2+a^2+2sa\cos\theta}$ runs over $(|s-a|,s+a)$, with $\sin\theta\,d\theta=\zeta\,d\zeta/(sa)$ and $\sin\theta=\Delta/(2sa)$, where
\[
\Delta=\Delta(s,a,\zeta):=\big[\big((s+a)^2-\zeta^2\big)\big(\zeta^2-(s-a)^2\big)\big]^{1/2}
\]
is symmetric in $(s,a,\zeta)$. Hence, by Tonelli's theorem (the condition $|s-a|<\zeta<s+a$ is symmetric in $s$ and $\zeta$),
\begin{align*}
\int_0^\infty (M_hg)(x)\,x^{\alpha}dx
&=2c_\alpha\int_0^\infty\!\!\int_{|s-a|}^{s+a} g(\zeta^2)\Big(\frac{\Delta}{2sa}\Big)^{2\alpha-1}\frac{\zeta}{sa}\,s^{2\alpha+1}\,d\zeta\,ds\\
&=\frac{2c_\alpha}{(2a)^{2\alpha-1}a}\int_0^\infty g(\zeta^2)\,\zeta\Big[\int_{|\zeta-a|}^{\zeta+a}\Delta^{2\alpha-1}\,s\,ds\Big]d\zeta .
\end{align*}
The inner integral is computed by the reverse substitution $s^2=\zeta^2+a^2+2\zeta a\cos\varphi$, $0<\varphi<\pi$, for which $s\,ds=\zeta a\sin\varphi\,d\varphi$ and $\Delta(s,a,\zeta)=2\zeta a\sin\varphi$:
\[
\int_{|\zeta-a|}^{\zeta+a}\Delta^{2\alpha-1}\,s\,ds=(2\zeta a)^{2\alpha-1}\zeta a\int_0^\pi\sin^{2\alpha}\varphi\,d\varphi=\frac{(2\zeta a)^{2\alpha-1}\zeta a}{c_\alpha}.
\]
Inserting this gives
\[
\int_0^\infty (M_hg)(x)x^{\alpha}dx=2\int_0^\infty g(\zeta^2)\zeta^{2\alpha+1}d\zeta=\int_0^\infty g(z)z^{\alpha}dz ,
\]
which is \eqref{masspres}. (For $\alpha=0$ the exponent $2\alpha-1=-1$ produces integrable endpoint singularities only; the line $x=h$ is a null set.)

Finally, $T^{(\alpha)}_{h}(f;x)\to f(h)$ as $x\to0^+$ for continuous $f$, so
\[
e^{-x/2}\big|T^{(\alpha)}_{h}(f;x)\big|\to e^{h/2}\,\big|e^{-h/2}f(h)\big| ;
\]
choosing $f$ with $|e^{-h/2}f(h)|=\|f\|_{\mathcal{L}^\infty_{\alpha}}$ shows that $e^{h/2}$ is the norm of $T^{(\alpha)}_h$ on $\mathcal{L}^\infty_{\alpha}$.
\end{proof}
\begin{lemma}[Translation multiplier]\label{lemma23}
Let $\alpha\ge0$, $1\le p\le\infty$, and $f\in\mathcal L_\alpha^p$. Then
\[
\widehat{(T_h^{(\alpha)}f)}_\alpha(n)=R_n^{(\alpha)}(h)\widehat f_\alpha(n),
\qquad n\in\mathbb N_0,\quad h\ge0.
\]
\end{lemma}
\begin{proof}
For a classical Laguerre polynomial, the product formula \eqref{pf} and the definition \eqref{tr-op} give
\[
T_h^{(\alpha)}{\rm L}_n^{(\alpha)}=R_n^{(\alpha)}(h){\rm L}_n^{(\alpha)};
\]
the constants agree because
\[
\frac{\Gamma(\alpha+1)2^\alpha}{\sqrt{2\pi}}\,
\frac{\sqrt\pi}{2^{\alpha-1/2}}=\Gamma(\alpha+1).
\]
For completeness, we extend the identity to all the stated spaces, including $p=\infty$; cf.\ \cite[Corollary~2]{GM1}. The case $h=0$ is immediate. For $h>0$, changing variables $z=x+h+2\sqrt{xh}\cos\theta$ in \eqref{tr-op} yields
\[
T_h^{(\alpha)}f(x)=\int_0^\infty f(z)K_h(x,z)e^{-z}z^\alpha\,dz.
\]
On $|\sqrt x-\sqrt h|<\sqrt z<\sqrt x+\sqrt h$, put
\[
d_h(x,z)=\frac12\left[2(xh+xz+hz)-x^2-h^2-z^2\right]^{1/2}.
\]
The kernel is zero outside this region and equals
\[
K_h(x,z)=\frac{\Gamma(\alpha+1)2^{\alpha-1}}{\sqrt{2\pi}}
\frac{e^{(x+h+z)/2}}{(xhz)^\alpha}
J_{\alpha-1/2}(d_h(x,z))d_h(x,z)^{\alpha-1/2}
\]
inside it. Thus $K_h(x,z)=K_h(z,x)$; this is also the kernel representation in \cite[Section~2]{AV}. If $T_h^+$ denotes the operator with kernel $|K_h|$, the proof of Lemma \ref{Tnorm}, including \eqref{Tpointwise}, gives $\|T_h^+f\|_{\mathcal L_\alpha^p}\le e^{h/2}\|f\|_{\mathcal L_\alpha^p}$. Since $L_n^{(\alpha)}\in\mathcal L_\alpha^{p'}$, H\"older's inequality bounds the absolute double integral for pairing $T_h f$ with $L_n^{(\alpha)}$ by
\[
e^{h/2}\|f\|_{\mathcal L_\alpha^p}\|L_n^{(\alpha)}\|_{\mathcal L_\alpha^{p'}}<\infty.
\]
Fubini's theorem and symmetry therefore give, for every $1\le p\le\infty$,
\begin{align*}
\widehat{(T_h^{(\alpha)}f)}_\alpha(n)
&=\int_0^\infty f(x)\,T_h^{(\alpha)}L_n^{(\alpha)}(x)e^{-x}x^\alpha\,dx\\
&=R_n^{(\alpha)}(h)\widehat f_\alpha(n).
\end{align*}
\end{proof}
\begin{remark}\label{notcontraction}
The multiplier $R_{n}^{(\alpha)}(h)={}_1F_1(-n;\alpha+1;h)$ is a polynomial of degree $n$ in $h$. Consequently:
\begin{itemize}
\item[(a)] the family $\{T_h^{(\alpha)}\}_{h\ge0}$ is not a semigroup. Indeed,
\[
R_1^{(\alpha)}(h_1)R_1^{(\alpha)}(h_2)-R_1^{(\alpha)}(h_1+h_2)
=\frac{h_1h_2}{(\alpha+1)^2}.
\]
The multiplier identity and Plancherel's theorem also give
\[
\|T_h^{(\alpha)}\|_{\mathcal L_\alpha^2\to\mathcal L_\alpha^2}
=\sup_n|R_n^{(\alpha)}(h)|\le e^{h/2}.
\]
These norms are unbounded as $h\to\infty$, already by the mode $n=1$, so the operators are not contractions for all $h$;
\item[(b)] for every fixed $n\ge1$, $|1-R_{n}^{(\alpha)}(h)|\to\infty$ as $h\to\infty$; already
$\big|1-R_{1}^{(\alpha)}(h)\big|=\frac{h}{\alpha+1}$.
\end{itemize}
For this reason all statements below involving the modulus of continuity are restricted to $0<t\le1$; see Remark \ref{whysmall}.
\end{remark}
It may be useful to summarise at this point the ways in which $T^{(\alpha)}_h$ departs from the Euclidean translation, since it is these that dictate the form of the results below. All entries in the last column are proved in, or immediate from, Lemma \ref{Tnorm} and Remarks \ref{notcontraction} and \ref{nosubadd}.
\begin{center}\footnotesize\setlength{\tabcolsep}{4pt}
\begin{tabular}{@{}lll@{}}
\toprule
 & $\tau_hf=f(\cdot+h)$ on $\mathbb{R}$ & $T^{(\alpha)}_h$ on $(0,\infty)$\\
\midrule
multiplier & $e^{i\lambda h}$ & $R^{(\alpha)}_n(h)={}_1F_1(-n;\alpha+1;h)$\\
its modulus & $\equiv1$ & $\le e^{h/2}$, unbounded in $h$\\
semigroup in $h$ & yes & no\\
norm on $L^p$ & $=1$ for every $h$ & $\le e^{h/2}$; $=e^{h/2}$ for $p=\infty$\\
subadditive for all $t>0$ & yes & no\\
\bottomrule
\end{tabular}
\end{center}
\subsection{Moduli of smoothness with respect to the Laguerre translation}
Throughout, for $\alpha\ge0$, $n\in\mathbb{N}_0$ and $h\ge0$ we write
\begin{equation}\label{rhodef}
\rho^{(\alpha)}_n(h):=1-R^{(\alpha)}_n(h)
\end{equation}
for the multiplier of the difference operator $I-T^{(\alpha)}_h$. Set
\begin{equation}\label{Bp}
B_{p,\alpha}:=\sup_{0\le h\le 1}\big\|T_{h}^{(\alpha)}\big\|_{\mathcal{L}^p_{\alpha}\to\mathcal{L}^p_{\alpha}}.
\end{equation}
By Lemma \ref{Tnorm}, $B_{p,\alpha}\le e^{1/2}$ for every $1\le p\le\infty$; at $p=2$ one has more precisely $B_{2,\alpha}=\sup_{0\le h\le1}\sup_n|R^{(\alpha)}_n(h)|$ by Remark \ref{notcontraction}(a). Norm estimates for $T^{(\alpha)}_h$ on a wider scale of weighted spaces are due to G\"orlich and Markett \cite{GM}; for generalized translations attached to singular Sturm--Liouville operators, of which the Laguerre and Hankel translations are the prototypes, see also Markett \cite{Mark84}.
Let $f\in\mathcal{L}^p_{\alpha}(\mathbb{R}^{+})$, $\alpha\ge0$, $1\le p\le\infty$, and let $r\in\mathbb{N}$. For $h\ge0$ put
\[
\Delta^{r}_{h}f:=\big(I-T^{(\alpha)}_{h}\big)^{r}f=\sum_{j=0}^{r}(-1)^{j}\binom{r}{j}\big(T^{(\alpha)}_{h}\big)^{j}f ,
\]
where $\big(T^{(\alpha)}_{h}\big)^{j}$ denotes the $j$-fold composition. For $r=1$ nothing is required beyond the definition of $T^{(\alpha)}_h$: $\Delta^{1}_{h}f=f-T^{(\alpha)}_{h}f$, and $\Delta^{1}_{h}f\in\mathcal{L}^p_{\alpha}$ precisely when $\|T^{(\alpha)}_{h}f-f\|_{\mathcal{L}^p_{\alpha}}<\infty$. For $r\ge2$ the compositions $\big(T^{(\alpha)}_{h}\big)^{j}$ are meaningful because $T^{(\alpha)}_{h}$ maps $\mathcal{L}^p_{\alpha}$ into itself (Lemma \ref{Tnorm}), and $\|\Delta^{r}_{h}f\|_{\mathcal{L}^p_{\alpha}}\le(1+B_{p,\alpha})^{r}\|f\|_{\mathcal{L}^p_{\alpha}}\le(1+e^{1/2})^{r}\|f\|_{\mathcal{L}^p_{\alpha}}$ for $0\le h\le1$. Iterating Lemma \ref{lemma23} gives the multiplier identity
\begin{equation}\label{diffmult}
L\big\{\Delta^{r}_{h}f\big\}(n)=\widehat{f}_{\alpha}(n)\,\big(\rho^{(\alpha)}_n(h)\big)^{r},\qquad n\in\mathbb{N}_0.
\end{equation}
\begin{definition}\label{moduli}
The \emph{modulus of smoothness of order $r$} of $f\in\mathcal{L}^p_{\alpha}(\mathbb{R}^{+})$ with respect to the Laguerre translation is
 \begin{equation}\label{Modulus}
 \Omega_r(L; f;t)_p := \sup_{0\leq h \leq t} \big\|\Delta^{r}_{h}f\big\|_{\mathcal{L}^p_{\alpha}},\qquad 0<t\le1,\qquad \Omega_r(L;f;0)_p:=0.
\end{equation}
For $r=1$ we recover the modulus of continuity, and we write $\Omega(L;f;t)_p$ for $\Omega_1(L;f;t)_p$.
\end{definition}
 \begin{property}\label{prop} Let $\alpha\ge0$, $1\le p\le\infty$ and $r\in\mathbb{N}$. The following properties of $\Omega_r(L;f;t)_p$ are valid:
\begin{itemize}
 \item[I)] $\Omega_r(L; f;t)_p$ is non-negative and $\Omega_r(L; f;0)_p=0$;
 \item[II)] $\Omega_r(L; f;t)_p$ is non-decreasing on $[0,1]$;
 \item[III)] if $p<\infty$, then $\Omega_r(L; f;t)_p\to0$ as $t\to0^+$, i.e. $\Omega_r(L;f;\cdot)_p$ is continuous at $t=0$.
\end{itemize}\end{property}
\begin{proof}
 Property I) follows from \eqref{Modulus}; II) follows from the fact that for larger $t$ the supremum is taken over a wider set of values of $h$.
Let us prove III). Fix $\varepsilon>0$. By Lemma \ref{dense} choose $P=\sum_{k=0}^{N}a_kL_k^{(\alpha)}$ with $\|f-P\|_{\mathcal{L}^p_{\alpha}}<\varepsilon$. For $0\le h\le 1$,
\[
\big\|\Delta^{r}_{h}f\big\|_{\mathcal{L}^p_{\alpha}}\le \big\|\Delta^{r}_{h}(f-P)\big\|_{\mathcal{L}^p_{\alpha}}+\big\|\Delta^{r}_{h}P\big\|_{\mathcal{L}^p_{\alpha}}
\le (1+B_{p,\alpha})^{r}\varepsilon+\big\|\Delta^{r}_{h}P\big\|_{\mathcal{L}^p_{\alpha}}.
\]
By Lemma \ref{lemma23} applied to $L^{(\alpha)}_k$, $T^{(\alpha)}_{h}L^{(\alpha)}_k=R^{(\alpha)}_k(h)L^{(\alpha)}_k$, whence $\Delta^{r}_{h}P=\sum_{k=0}^{N}a_k\big(\rho^{(\alpha)}_k(h)\big)^{r}L_k^{(\alpha)}$; each $\rho^{(\alpha)}_k$ is a polynomial with $\rho^{(\alpha)}_k(0)=0$, so the last term tends to $0$ as $h\to0^+$, uniformly for $0\le h\le t$ as $t\to 0^+$. Hence $\limsup_{t\to0^+}\Omega_r(L;f;t)_p\le(1+B_{p,\alpha})^{r}\varepsilon$, and $\varepsilon>0$ was arbitrary.
\end{proof}
\begin{remark}[The globally extended modulus]\label{nosubadd}
Although we use \eqref{Modulus} only for $0\le t\le1$, the same expression is finite for every fixed $t>0$ by Lemma \ref{Tnorm}. Its extension to all positive $t$ need not be subadditive. For $r=1$, $\alpha=0$, and $f={\rm L}_2^{(0)}$, Plancherel's identity and Lemma \ref{lemma23} give
\[
\Omega_1(L;f;t)_2=\sup_{0\le h\le t}|2h-h^2/2|.
\]
Consequently $\Omega_1(L;f;10)_2=30$, whereas $2\Omega_1(L;f;5)_2=5$. This is a statement about the global extension and makes no assertion about subadditivity when the two steps and their sum lie in $[0,1]$. All subsequent arguments use the properties already proved above and do not require subadditivity.
\end{remark}
\begin{remark}\label{whysmall}
By Remark \ref{notcontraction}(b), for $f=L_1^{(\alpha)}$ one has $\rho^{(\alpha)}_1(h)=\frac{h}{\alpha+1}$ and hence, extending \eqref{Modulus} to all $t>0$, $\Omega_r(L;f;t)_2=\big(\frac{t}{\alpha+1}\big)^{r}$. Consequently a condition $\Omega_r(L;f;t)_p=O(t^{\beta})$ with $\beta<r$ imposed for \emph{all} $t>0$ would force $\widehat f_\alpha(n)=0$ for every $n\ge1$, i.e.\ it would only be satisfied by the constants. This is why $t$ is restricted to $(0,1]$ in \eqref{Modulus} and in Definition \ref{lip} below.
\end{remark}
\section{Two-sided estimates for the multiplier}\label{sec3}
Throughout this section $\alpha\ge0$, $\rho^{(\alpha)}_n$ is as in \eqref{rhodef}, and
\[
\varepsilon_\alpha:=\frac{\alpha+2}{2\sqrt{e}},\qquad \eta_\alpha:=\min\{\varepsilon_\alpha,1\}.
\]
We collect here the estimates for $\rho^{(\alpha)}_n$ which replace the multiplier $\min\{1,(\lambda t)^2\}$ of Proposition \ref{pro1}; Figure \ref{fig:profile} shows the behaviour they describe.
\begin{lemma}\label{lemA}
Let $\alpha\ge0$, $n\ge1$ and $0\le h\le1$. Then
\begin{equation}\label{Ai}
\big|\rho^{(\alpha)}_n(h)\big|\le\frac{\sqrt{e}}{\alpha+1}\,nh,
\end{equation}
and
\begin{equation}\label{Aii}
\rho^{(\alpha)}_n(h)\ge\frac{nh}{2(\alpha+1)}\qquad\text{whenever } nh\le\varepsilon_\alpha.
\end{equation}
Moreover $\rho^{(\alpha)}_0\equiv0$ and $\big|\rho^{(\alpha)}_n(h)\big|\le 1+\sqrt{e}$ for all $n\ge0$, $0\le h\le1$.
\end{lemma}
\begin{proof}
From $\frac{d}{dz}{}_1F_1(a;b;z)=\frac{a}{b}\,{}_1F_1(a+1;b+1;z)$ we get
\begin{equation}\label{derivative}
\frac{d}{dh}R_{n}^{(\alpha)}(h)=-\frac{n}{\alpha+1}\,{}_1F_1(-(n-1);\alpha+2;h)=-\frac{n}{\alpha+1}\,R_{n-1}^{(\alpha+1)}(h),
\end{equation}
and therefore, since $R_{n}^{(\alpha)}(0)=1$,
\begin{equation}\label{integralform}
\rho^{(\alpha)}_n(h)=\frac{n}{\alpha+1}\int_{0}^{h}R_{n-1}^{(\alpha+1)}(s)\,ds.
\end{equation}
Applying \eqref{est2} with the order $\alpha+1\ (\ge0)$ gives $\big|R_{n-1}^{(\alpha+1)}(s)\big|\le e^{s/2}$, whence for $0\le h\le1$
\[
\big|\rho^{(\alpha)}_n(h)\big|\le\frac{n}{\alpha+1}\int_{0}^{h}e^{s/2}ds=\frac{2n\big(e^{h/2}-1\big)}{\alpha+1}\le\frac{n\,h\,e^{h/2}}{\alpha+1}\le\frac{\sqrt{e}}{\alpha+1}\,nh,
\]
where we used $e^{u}-1\le u e^{u}$ with $u=h/2$. This proves \eqref{Ai}.
Now assume $nh\le\varepsilon_\alpha$ and let $0\le s\le h\ (\le1)$. If $n=1$ then $R^{(\alpha+1)}_{0}\equiv1$, and if $n\ge2$ then, applying \eqref{Ai} with $\alpha+1$ and $n-1$ in place of $\alpha$ and $n$,
\[
\big|1-R_{n-1}^{(\alpha+1)}(s)\big|\le\frac{\sqrt{e}}{\alpha+2}(n-1)s\le\frac{\sqrt{e}}{\alpha+2}\,ns\le\frac{\sqrt{e}}{\alpha+2}\,\varepsilon_\alpha=\frac12 ,
\]
so that $R_{n-1}^{(\alpha+1)}(s)\ge\frac12$ on $[0,h]$. Together with \eqref{integralform} this yields \eqref{Aii}. The last two assertions are immediate from $R^{(\alpha)}_0\equiv1$ and \eqref{est2}.
\end{proof}
The next lemma is the only place where a non-elementary property of the Laguerre polynomials is used; it supplies the information in the range $nh\gtrsim1$, in which Lemma \ref{lemA} gives nothing.
\begin{lemma}\label{lemB}
Let $\alpha\ge0$. There is a constant $C_\alpha$ such that
\begin{equation}\label{Bdecay}
\big|R_{n}^{(\alpha)}(h)\big|\le C_\alpha\,(nh)^{-\frac{\alpha}{2}-\frac14},\qquad n\ge1,\quad \frac1n\le h\le1.
\end{equation}
Consequently there is $A_\alpha\ge1$, depending only on $\alpha$, such that
\begin{equation}\label{Blower}
\big|\rho_n^{(\alpha)}(h)\big|\ge\frac12,
\qquad n\ge1,\quad 0<h\le1,\quad nh\ge A_\alpha.
\end{equation}
\end{lemma}
\begin{proof}
We use the classical oscillatory estimate for Laguerre polynomials; see Askey and Wainger \cite{AW} and Muckenhoupt \cite{Muck}. In the normalisation used here, Gasper and Trebels \cite[formula~(8)]{GTL} state
\begin{equation}\label{oscill}
\begin{gathered}
e^{-x/2}|R_n^{(\alpha)}(x)|\le C_\alpha(x\nu_n)^{-\alpha/2-1/4},\\
\frac1{2\nu_n}\le x\le\frac{3\nu_n}{4},\qquad \nu_n=4n+2\alpha+2.
\end{gathered}
\end{equation}
For $n\ge1$ and $1/n\le h\le1$, this range contains $h$, and $\nu_n\ge4n$. Multiplying \eqref{oscill} by $e^{h/2}\le\sqrt e$ and absorbing fixed factors into $C_\alpha$ proves \eqref{Bdecay}. We may enlarge $C_\alpha$ so that $C_\alpha\ge1$.
Finally choose $A_\alpha\ge1$ so large that $C_\alpha A_\alpha^{-\alpha/2-1/4}\le\frac12$. If $nh\ge A_\alpha$ and $0<h\le1$, then $h\ge\frac{A_\alpha}{n}\ge\frac1n$, so \eqref{Bdecay} applies and gives $\big|R_{n}^{(\alpha)}(h)\big|\le\frac12$, whence \eqref{Blower}.
\end{proof}
\begin{lemma}\label{lemC}
Let $\alpha\ge0$. There is a constant $c_\alpha>0$, depending only on $\alpha$, such that for every $q\ge1$, every $n\ge0$ and every $0<t\le1$,
\begin{equation}\label{Cmean}
\left(\frac1t\int_{0}^{t}\big|\rho^{(\alpha)}_n(h)\big|^{q}\,dh\right)^{1/q}\ \ge\ c_\alpha\,\min\{1,nt\}.
\end{equation}
In particular,
\begin{equation}\label{Csup}
c_\alpha\min\{1,nt\}\ \le\ \sup_{0\le h\le t}\big|\rho^{(\alpha)}_n(h)\big|\ \le\ c'_\alpha\min\{1,nt\},
\end{equation}
where $c'_\alpha:=1+\sqrt{e}$.
\end{lemma}
\begin{proof}
For $n=0$ both sides of \eqref{Cmean} vanish, so let $n\ge1$. We distinguish three cases; recall $\eta_\alpha=\min\{\varepsilon_\alpha,1\}\le1$ and $A_\alpha\ge1$.
\emph{Case 1: $nt\le\eta_\alpha$.} For $h\in[\frac t2,t]$ we have $h\le1$ and $nh\le nt\le\varepsilon_\alpha$, so \eqref{Aii} gives $\rho^{(\alpha)}_n(h)\ge\frac{nh}{2(\alpha+1)}\ge\frac{nt}{4(\alpha+1)}$. Hence
\[
\frac1t\int_{0}^{t}\big|\rho^{(\alpha)}_n\big|^{q}\,dh\ge\frac12\left(\frac{nt}{4(\alpha+1)}\right)^{q},
\]
and, since $2^{-1/q}\ge\frac12$ for $q\ge1$, the left-hand side of \eqref{Cmean} is at least $\frac{nt}{8(\alpha+1)}=\frac{\min\{1,nt\}}{8(\alpha+1)}$.
\emph{Case 2: $\eta_\alpha<nt\le2A_\alpha$.} Let $J:=\big[\frac{\eta_\alpha}{2n},\frac{\eta_\alpha}{n}\big]$. Then $J\subset[0,t]$ because $\frac{\eta_\alpha}{n}<t$, and for $h\in J$ we have $h\le\frac{\eta_\alpha}{n}\le1$ and $nh\le\eta_\alpha\le\varepsilon_\alpha$, so \eqref{Aii} gives $\rho^{(\alpha)}_n(h)\ge\frac{nh}{2(\alpha+1)}\ge\frac{\eta_\alpha}{4(\alpha+1)}$. Since $\frac{|J|}{t}=\frac{\eta_\alpha}{2nt}\ge\frac{\eta_\alpha}{4A_\alpha}$,
\[
\frac1t\int_{0}^{t}\big|\rho^{(\alpha)}_n\big|^{q}dh\ \ge\ \frac{\eta_\alpha}{4A_\alpha}\left(\frac{\eta_\alpha}{4(\alpha+1)}\right)^{q},
\]
and, using $x^{1/q}\ge x$ for $0\le x\le1$ and $q\ge1$, the left-hand side of \eqref{Cmean} is at least $\frac{\eta_\alpha}{4A_\alpha}\cdot\frac{\eta_\alpha}{4(\alpha+1)}\ge\frac{\eta^2_\alpha}{16A_\alpha(\alpha+1)}\min\{1,nt\}$.
\emph{Case 3: $nt>2A_\alpha$.} The set $\{h\in[0,t]:nh\ge A_\alpha\}=[\frac{A_\alpha}{n},t]$ has measure $t-\frac{A_\alpha}{n}\ge\frac t2$, and on it \eqref{Blower} gives $\big|\rho^{(\alpha)}_n(h)\big|\ge\frac12$. Hence $\frac1t\int_{0}^{t}\big|\rho^{(\alpha)}_n\big|^{q}dh\ge\frac12\cdot2^{-q}$ and the left-hand side of \eqref{Cmean} is at least $\frac14=\frac14\min\{1,nt\}$.
Thus \eqref{Cmean} holds with $c_\alpha=\min\big\{\frac{1}{8(\alpha+1)},\frac{\eta^2_\alpha}{16A_\alpha(\alpha+1)},\frac14\big\}$. The lower bound in \eqref{Csup} follows from \eqref{Cmean} because the supremum dominates every mean; the upper bound follows from Lemma \ref{lemA}, since
$$
\sup_{h\le t}|\rho^{(\alpha)}_n(h)|\le\frac{\sqrt e}{\alpha+1}nt
$$
if $nt\le1$ and $\sup_{h\le t}|\rho^{(\alpha)}_n(h)|\le1+\sqrt e$ always.
\end{proof}
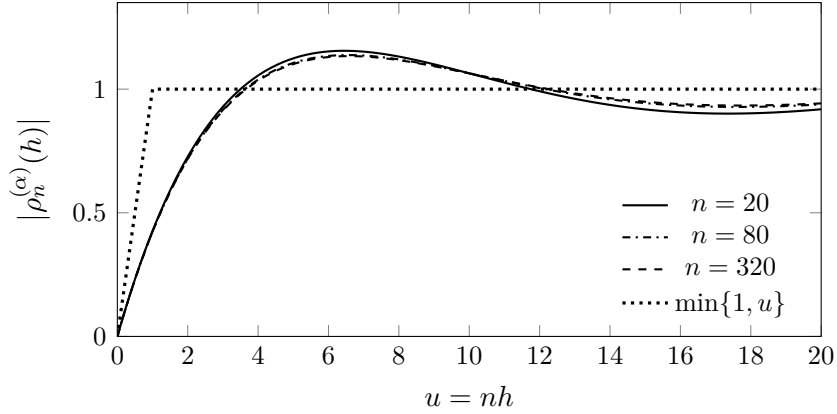
\begin{figure}[ht]
\begin{center}
\begin{tikzpicture}
\begin{axis}[width=0.86\textwidth,height=6cm,xlabel={$u=nh$},ylabel={$\big|\rho^{(\alpha)}_n(h)\big|$},
 xmin=0,xmax=20,ymin=0,ymax=1.35,legend pos=south east,
 legend style={font=\footnotesize,draw=none,fill=none},
 tick label style={font=\footnotesize},label style={font=\small}]
\addplot[thick,mark=none] coordinates {(0.001,0.0005) (0.169,0.0823) (0.337,0.1598) (0.505,0.2331) (0.673,0.3025) (0.841,0.3680) (1.009,0.4299) (1.177,0.4882) (1.345,0.5431) (1.514,0.5947) (1.682,0.6432) (1.850,0.6887) (2.018,0.7314) (2.186,0.7712) (2.354,0.8084) (2.522,0.8431) (2.690,0.8754) (2.858,0.9053) (3.026,0.9331) (3.194,0.9587) (3.362,0.9823) (3.530,1.0040) (3.698,1.0238) (3.866,1.0419) (4.034,1.0584) (4.202,1.0733) (4.371,1.0866) (4.539,1.0985) (4.707,1.1091) (4.875,1.1184) (5.043,1.1265) (5.211,1.1335) (5.379,1.1393) (5.547,1.1442) (5.715,1.1480) (5.883,1.1510) (6.051,1.1531) (6.219,1.1545) (6.387,1.1551) (6.555,1.1549) (6.723,1.1542) (6.891,1.1528) (7.059,1.1509) (7.228,1.1485) (7.396,1.1455) (7.564,1.1422) (7.732,1.1384) (7.900,1.1343) (8.068,1.1298) (8.236,1.1251) (8.404,1.1200) (8.572,1.1148) (8.740,1.1093) (8.908,1.1036) (9.076,1.0978) (9.244,1.0918) (9.412,1.0857) (9.580,1.0796) (9.748,1.0733) (9.916,1.0670) (10.085,1.0607) (10.253,1.0544) (10.421,1.0481) (10.589,1.0418) (10.757,1.0355) (10.925,1.0293) (11.093,1.0231) (11.261,1.0170) (11.429,1.0110) (11.597,1.0051) (11.765,0.9993) (11.933,0.9937) (12.101,0.9881) (12.269,0.9827) (12.437,0.9774) (12.605,0.9723) (12.773,0.9674) (12.942,0.9625) (13.110,0.9579) (13.278,0.9534) (13.446,0.9491) (13.614,0.9450) (13.782,0.9410) (13.950,0.9373) (14.118,0.9337) (14.286,0.9303) (14.454,0.9271) (14.622,0.9241) (14.790,0.9213) (14.958,0.9186) (15.126,0.9162) (15.294,0.9139) (15.462,0.9118) (15.630,0.9100) (15.799,0.9083) (15.967,0.9067) (16.135,0.9054) (16.303,0.9043) (16.471,0.9033) (16.639,0.9025) (16.807,0.9019) (16.975,0.9014) (17.143,0.9012) (17.311,0.9010) (17.479,0.9011) (17.647,0.9013) (17.815,0.9016) (17.983,0.9022) (18.151,0.9028) (18.319,0.9036) (18.487,0.9046) (18.656,0.9056) (18.824,0.9069) (18.992,0.9082) (19.160,0.9096) (19.328,0.9112) (19.496,0.9129) (19.664,0.9147) (19.832,0.9166) (20.000,0.9187)}; \addlegendentry{$n=20$}
\addplot[thick,dashdotted,mark=none] coordinates {(0.001,0.0005) (0.169,0.0822) (0.337,0.1595) (0.505,0.2324) (0.673,0.3013) (0.841,0.3662) (1.009,0.4274) (1.177,0.4849) (1.345,0.5390) (1.514,0.5898) (1.682,0.6374) (1.850,0.6820) (2.018,0.7237) (2.186,0.7627) (2.354,0.7990) (2.522,0.8328) (2.690,0.8642) (2.858,0.8933) (3.026,0.9202) (3.194,0.9450) (3.362,0.9679) (3.530,0.9889) (3.698,1.0081) (3.866,1.0257) (4.034,1.0416) (4.202,1.0560) (4.371,1.0689) (4.539,1.0805) (4.707,1.0908) (4.875,1.0999) (5.043,1.1079) (5.211,1.1147) (5.379,1.1206) (5.547,1.1254) (5.715,1.1294) (5.883,1.1326) (6.051,1.1349) (6.219,1.1366) (6.387,1.1375) (6.555,1.1378) (6.723,1.1375) (6.891,1.1367) (7.059,1.1353) (7.228,1.1335) (7.396,1.1312) (7.564,1.1286) (7.732,1.1256) (7.900,1.1222) (8.068,1.1186) (8.236,1.1147) (8.404,1.1106) (8.572,1.1062) (8.740,1.1017) (8.908,1.0970) (9.076,1.0922) (9.244,1.0872) (9.412,1.0822) (9.580,1.0770) (9.748,1.0719) (9.916,1.0667) (10.085,1.0614) (10.253,1.0562) (10.421,1.0509) (10.589,1.0457) (10.757,1.0405) (10.925,1.0354) (11.093,1.0303) (11.261,1.0253) (11.429,1.0203) (11.597,1.0154) (11.765,1.0107) (11.933,1.0060) (12.101,1.0014) (12.269,0.9970) (12.437,0.9926) (12.605,0.9884) (12.773,0.9843) (12.942,0.9804) (13.110,0.9766) (13.278,0.9729) (13.446,0.9693) (13.614,0.9660) (13.782,0.9627) (13.950,0.9596) (14.118,0.9566) (14.286,0.9538) (14.454,0.9512) (14.622,0.9487) (14.790,0.9463) (14.958,0.9441) (15.126,0.9421) (15.294,0.9402) (15.462,0.9384) (15.630,0.9368) (15.799,0.9353) (15.967,0.9340) (16.135,0.9328) (16.303,0.9317) (16.471,0.9308) (16.639,0.9300) (16.807,0.9294) (16.975,0.9289) (17.143,0.9285) (17.311,0.9282) (17.479,0.9281) (17.647,0.9280) (17.815,0.9281) (17.983,0.9283) (18.151,0.9286) (18.319,0.9290) (18.487,0.9296) (18.656,0.9302) (18.824,0.9309) (18.992,0.9317) (19.160,0.9325) (19.328,0.9335) (19.496,0.9346) (19.664,0.9357) (19.832,0.9369) (20.000,0.9382)}; \addlegendentry{$n=80$}
\addplot[thick,dashed,mark=none] coordinates {(0.001,0.0005) (0.169,0.0822) (0.337,0.1594) (0.505,0.2323) (0.673,0.3010) (0.841,0.3658) (1.009,0.4268) (1.177,0.4841) (1.345,0.5380) (1.514,0.5886) (1.682,0.6360) (1.850,0.6804) (2.018,0.7219) (2.186,0.7606) (2.354,0.7967) (2.522,0.8302) (2.690,0.8614) (2.858,0.8903) (3.026,0.9170) (3.194,0.9417) (3.362,0.9644) (3.530,0.9852) (3.698,1.0043) (3.866,1.0217) (4.034,1.0375) (4.202,1.0518) (4.371,1.0646) (4.539,1.0761) (4.707,1.0864) (4.875,1.0954) (5.043,1.1033) (5.211,1.1102) (5.379,1.1160) (5.547,1.1209) (5.715,1.1249) (5.883,1.1281) (6.051,1.1305) (6.219,1.1322) (6.387,1.1332) (6.555,1.1336) (6.723,1.1335) (6.891,1.1327) (7.059,1.1315) (7.228,1.1298) (7.396,1.1277) (7.564,1.1252) (7.732,1.1224) (7.900,1.1193) (8.068,1.1158) (8.236,1.1121) (8.404,1.1082) (8.572,1.1040) (8.740,1.0997) (8.908,1.0953) (9.076,1.0907) (9.244,1.0859) (9.412,1.0811) (9.580,1.0763) (9.748,1.0713) (9.916,1.0664) (10.085,1.0614) (10.253,1.0564) (10.421,1.0514) (10.589,1.0464) (10.757,1.0414) (10.925,1.0365) (11.093,1.0317) (11.261,1.0269) (11.429,1.0222) (11.597,1.0175) (11.765,1.0130) (11.933,1.0085) (12.101,1.0042) (12.269,0.9999) (12.437,0.9958) (12.605,0.9918) (12.773,0.9879) (12.942,0.9841) (13.110,0.9805) (13.278,0.9770) (13.446,0.9736) (13.614,0.9704) (13.782,0.9673) (13.950,0.9643) (14.118,0.9615) (14.286,0.9588) (14.454,0.9563) (14.622,0.9539) (14.790,0.9516) (14.958,0.9495) (15.126,0.9475) (15.294,0.9457) (15.462,0.9440) (15.630,0.9425) (15.799,0.9410) (15.967,0.9397) (16.135,0.9386) (16.303,0.9376) (16.471,0.9367) (16.639,0.9359) (16.807,0.9352) (16.975,0.9347) (17.143,0.9343) (17.311,0.9340) (17.479,0.9338) (17.647,0.9337) (17.815,0.9337) (17.983,0.9339) (18.151,0.9341) (18.319,0.9344) (18.487,0.9349) (18.656,0.9354) (18.824,0.9360) (18.992,0.9367) (19.160,0.9374) (19.328,0.9383) (19.496,0.9392) (19.664,0.9402) (19.832,0.9413) (20.000,0.9424)}; \addlegendentry{$n=320$}
\addplot[dotted,very thick,domain=0:20,samples=400] {min(1,x)}; \addlegendentry{$\min\{1,u\}$}
\end{axis}
\end{tikzpicture}
\end{center}
\caption{The multiplier $\big|\rho^{(\alpha)}_n(h)\big|=\big|1-R^{(\alpha)}_n(h)\big|$ for $\alpha=1$, plotted against $u=nh$ for three values of $n$; the displayed range is $0<u\le20$, corresponding to $0<h\le20/n\le1$ for each of the three curves. The three curves nearly coincide, so that $\big|\rho^{(\alpha)}_n(h)\big|$ depends on $n$ and $h$ essentially through the product $nh$ alone: it grows linearly for $nh$ small and then oscillates about a height comparable to $1$. This is the qualitative content of \eqref{Ai}, \eqref{Aii} and \eqref{Blower}, and the reason the weight $\min\{1,nt\}$ appears throughout Section \ref{sec4}.}\label{fig:profile}
\end{figure}
\begin{remark}\label{uniformq}
The constant $c_\alpha$ in \eqref{Cmean} is independent of $q$. This gives constants with a simple dependence on the order in Section \ref{sec4}: by \eqref{diffmult} the multiplier of $\Delta^{r}_{h}$ is $\big(\rho^{(\alpha)}_n(h)\big)^{r}$, and \eqref{Cmean} applied with $q=rs$ gives, for every $s\ge1$ and $r\in\mathbb{N}$,
\begin{equation}\label{Crmean}
\frac1t\int_{0}^{t}\big|\rho^{(\alpha)}_n(h)\big|^{rs}\,dh\ \ge\ \big(c_\alpha\min\{1,nt\}\big)^{rs},\qquad 0<t\le1.
\end{equation}
Likewise, raising \eqref{Csup} to the power $r$,
\begin{equation}\label{Crsup}
\big(c_\alpha\min\{1,nt\}\big)^{r}\ \le\ \sup_{0\le h\le t}\big|\rho^{(\alpha)}_n(h)\big|^{r}\ \le\ \big(c'_\alpha\min\{1,nt\}\big)^{r},
\end{equation}
where $c'_\alpha=1+\sqrt e$.
\end{remark}
\section{Growth of the Laguerre transform}\label{sec4}
\begin{theorem}\label{estimates}
 Let $\alpha\geq 0$, $1\leq p\leq 2$, $r\in\mathbb{N}$, $0<t\le1$, and let $f\in\mathcal{L}^p_{\alpha}(\mathbb{R}^{+})$. Then, with $c_\alpha$ as in Lemma \ref{lemC} and $c_{\alpha,p}$ as in Theorem \ref{hyl},
 \begin{equation}\label{main1}
 \left(\sum_{n=0}^{\infty}(\lambda_{n}^{\alpha})^{\frac{2-p'}{2}}\min\left\{1, nt\right\}^{rp'}\big|\widehat{f}_{\alpha}(n)\big|^{p'}\right)^{1/p'}\leq \frac{c_{\alpha,p}}{c_\alpha^{\,r}}\ \Omega_r(L; f;t)_p,
 \end{equation}
 for $1<p\leq 2$, and
 \begin{equation}\label{main2}
 \sup_{n\ge0} \left\{(\lambda_{n}^{\alpha})^{-\frac{1}{2}}\min\left\{1, nt\right\}^{r}\big|\widehat{f}_{\alpha}(n)\big|\right\}\leq \frac{1}{c_\alpha^{\,r}\,\Gamma(\alpha+1)}\ \Omega_r(L; f;t)_1,
 \end{equation}
 for $p=1$.
\end{theorem}
\begin{proof}
We may assume $\Omega_r(L;f;t)_p<\infty$; then $g_h:=\Delta^{r}_{h}f\in\mathcal{L}^p_{\alpha}$ for every $h\in[0,t]$, by Lemma \ref{Tnorm}. \begin{equation}\label{diffcoef}
\widehat{(g_h)}_\alpha(n)=\widehat{f}_{\alpha}(n)\,\big(\rho^{(\alpha)}_n(h)\big)^{r},\qquad n\ge0.
\end{equation}
\emph{The case $p=1$.} Fix $n\ge0$. By Lemma \ref{lem1} applied to $g_h$ and by \eqref{diffcoef}, for every $0\le h\le t$,
\[
(\lambda_{n}^{\alpha})^{-1/2}\big|\widehat{f}_{\alpha}(n)\big|\,\big|\rho^{(\alpha)}_n(h)\big|^{r}\le\frac{1}{\Gamma(\alpha+1)}\|g_h\|_{\mathcal{L}^1_{\alpha}}\le\frac{1}{\Gamma(\alpha+1)}\Omega_r(L;f;t)_1.
\]
Taking the supremum over $h\in[0,t]$ on the left and using the lower bound in \eqref{Crsup},
\[
c_\alpha^{\,r}\min\{1,nt\}^{r}\,(\lambda_{n}^{\alpha})^{-1/2}\big|\widehat{f}_{\alpha}(n)\big|\le\frac{1}{\Gamma(\alpha+1)}\Omega_r(L;f;t)_1.
\]
Since $n$ was arbitrary, \eqref{main2} follows. (Note that here the two suprema commute, which is why no averaging is needed.)
\emph{The case $1<p\le2$.} The function $h\mapsto\|g_h\|_{\mathcal{L}^p_{\alpha}}$ is measurable on $[0,t]$, and it is bounded above by $\Omega_r(L;f;t)_p$. Hence, replacing the supremum by the mean and applying Theorem \ref{hyl} to $g_h$,
\[
\Omega_r(L;f;t)_p^{p'}\ \ge\ \frac1t\int_{0}^{t}\|g_h\|^{p'}_{\mathcal{L}^p_{\alpha}}\,dh\ \ge\ c_{\alpha,p}^{-p'}\,\frac1t\int_{0}^{t}\big\|\widehat{(g_h)}_\alpha\big\|^{p'}_{{l}^{p'}_{\alpha}}\,dh.
\]
By \eqref{diffcoef} and Tonelli's theorem the last integral equals
\begin{align*}
\frac1t\int_{0}^{t}\sum_{n=0}^{\infty}&(\lambda_{n}^{\alpha})^{\frac{2-p'}{2}}\big|\rho^{(\alpha)}_n(h)\big|^{rp'}\big|\widehat{f}_{\alpha}(n)\big|^{p'}dh\\
&=\sum_{n=0}^{\infty}(\lambda_{n}^{\alpha})^{\frac{2-p'}{2}}\big|\widehat{f}_{\alpha}(n)\big|^{p'}\left(\frac1t\int_{0}^{t}\big|\rho^{(\alpha)}_n(h)\big|^{rp'}dh\right).
\end{align*}
Since $p'\ge2\ge1$, \eqref{Crmean} with $s=p'$ bounds the inner mean from below by $\big(c_\alpha\min\{1,nt\}\big)^{rp'}$, and we obtain
\[
\Omega_r(L;f;t)_p^{p'}\ \ge\ \left(\frac{c_\alpha^{\,r}}{c_{\alpha,p}}\right)^{p'}\sum_{n=0}^{\infty}(\lambda_{n}^{\alpha})^{\frac{2-p'}{2}}\min\{1,nt\}^{rp'}\big|\widehat{f}_{\alpha}(n)\big|^{p'}.
\]
Raising both sides to the power $1/p'$ gives \eqref{main1}.
\end{proof}
\begin{remark}
At $p=p'=2$ the weight $(\lambda^{\alpha}_n)^{\frac{2-p'}{2}}$ in \eqref{main1} equals $1$. Any weight growing in $n$ would contradict \eqref{pe}, since $\lambda^{\alpha}_n\asymp n^{\alpha}$; Proposition \ref{sharp} below shows that the factor $\min\{1,nt\}^{r}$ is optimal as well.
\end{remark}
In the complementary range $p\ge2$ the same weighted sum bounds the modulus from \emph{above}.
\begin{theorem}\label{estimatesdual}
 Let $\alpha\geq 0$, $2\le p\le\infty$, $r\in\mathbb{N}$, $0<t\le1$, and let $f\in\mathcal{L}^p_{\alpha}(\mathbb{R}^{+})$. Then, with $c^{*}_{\alpha,p}$ as in Theorem \ref{hyldual} and $c'_\alpha=1+\sqrt e$,
 \begin{equation}\label{main3}
 \Omega_r(L; f;t)_p\ \leq\ c^{*}_{\alpha,p}\,(c'_\alpha)^{r}\left(\sum_{n=0}^{\infty}(\lambda_{n}^{\alpha})^{\frac{2-p'}{2}}\min\left\{1, nt\right\}^{rp'}\big|\widehat{f}_{\alpha}(n)\big|^{p'}\right)^{1/p'}.
 \end{equation}
\end{theorem}
\begin{proof}
Assume that the right-hand side of \eqref{main3} is finite and fix $h\in[0,t]$. By \eqref{diffmult}, the coefficients of $\Delta_h^r f$ are $\widehat f_\alpha(n)(\rho_n^{(\alpha)}(h))^r$. Applying the upper multiplier bound \eqref{Crsup} to each coefficient gives
\[
\big\|\widehat{\big(\Delta^{r}_{h}f\big)}_{\alpha}\big\|_{{l}^{p'}_{\alpha}}
\ \le\ (c'_\alpha)^{r}\left(\sum_{n=0}^{\infty}(\lambda_{n}^{\alpha})^{\frac{2-p'}{2}}\min\{1,nt\}^{rp'}\big|\widehat{f}_{\alpha}(n)\big|^{p'}\right)^{1/p'},
\]
and Theorem \ref{hyldual}\,{\rm(ii)} applied to $\Delta^{r}_{h}f\in\mathcal{L}^p_{\alpha}$ bounds $\|\Delta^{r}_{h}f\|_{\mathcal{L}^p_{\alpha}}$ by $c^{*}_{\alpha,p}$ times the left-hand side. The bound is independent of $h$, so taking the supremum over $h\in[0,t]$ gives \eqref{main3}.
\end{proof}
\begin{remark}\label{whyp2}
Theorems \ref{estimates} and \ref{estimatesdual} are mirror images of one another: one and the same weighted sum bounds $\Omega_r(L;f;t)_p$ from below when $1\le p\le2$, and from above when $2\le p\le\infty$. The exponent $p=2$ is the only one at which both are available, and there they combine into the equivalence of Theorem \ref{equiv} below.
The lower estimate cannot be reversed for $1\le p<2$, even when the order $r$ is fixed. Suppose such a reversal held with a constant independent of $f,t$. Take $t=1$ and $f=L_m^{(\alpha)}$, $m\ge1$. By \eqref{Crsup},
\[
\Omega_r(L;L_m^{(\alpha)};1)_p
=\sup_{0\le h\le1}|\rho_m^{(\alpha)}(h)|^r\|L_m^{(\alpha)}\|_{\mathcal L_\alpha^p}
\ge c_\alpha^r\|L_m^{(\alpha)}\|_{\mathcal L_\alpha^p}.
\]
The weighted coefficient expression for this single polynomial is
\[
(\lambda_m^\alpha)^{1/p'-1/2}
\asymp m^{-\alpha(1/p-1/2)},
\]
with $p'=\infty$ and $1/p'=0$ when $p=1$. The Cohen-type estimate of Markett \cite[formula~(1.13)]{Mark}, stated in Gasper and Trebels \cite[Remark~1]{GT2}, gives
\[
\left\|{\rm L}_m^{(\alpha)}\right\|_{\mathcal L_\alpha^p}
\ge c\,(m+1)^{(\alpha+1)/p-1/2},\qquad 1\le p<2.
\]
After multiplication by $(\lambda_m^\alpha)^{-1/2}$, this yields
\[
\|L_m^{(\alpha)}\|_{\mathcal L_\alpha^p}
\ge c\,m^{(\alpha+1)(1/p-1/2)}.
\]
The ratio of the proposed left- and right-hand sides is therefore bounded below by a positive multiple of
$m^{(2\alpha+1)(1/p-1/2)}$, which tends to infinity. For the broader multiplier framework, see also Part~I \cite{GT1}.
\end{remark}
For $p=2$ the inequality of Theorem \ref{estimates} can be reversed. Put
\begin{equation}\label{Sdef}
\mathcal{S}_r(f;t):=\left(\sum_{n=0}^{\infty}\min\{1,nt\}^{2r}\big|\widehat{f}_{\alpha}(n)\big|^{2}\right)^{1/2},\qquad 0<t\le1.
\end{equation}
\begin{theorem}\label{equiv}
Let $\alpha\ge0$, $r\in\mathbb{N}$, $f\in\mathcal{L}^2_{\alpha}(\mathbb{R}^{+})$ and $0<t\le1$. Then
\begin{equation}\label{equivineq}
c_\alpha^{\,r}\,\mathcal{S}_r(f;t)\ \le\ \Omega_r(L;f;t)_2\ \le\ (c'_\alpha)^{r}\,\mathcal{S}_r(f;t),\qquad c'_\alpha=1+\sqrt e.
\end{equation}
In particular $\Omega_r(L;f;t)_2\asymp\mathcal{S}_r(f;t)$, with constants depending only on $\alpha$ and $r$.
\end{theorem}
\begin{proof}
By the Plancherel identity \eqref{pe} and \eqref{diffmult}, for every $h\ge0$,
\begin{equation}\label{pch2}
\big\|\Delta^{r}_{h}f\big\|^{2}_{\mathcal{L}^2_{\alpha}}=\sum_{n=0}^{\infty}\big|\rho^{(\alpha)}_n(h)\big|^{2r}\big|\widehat{f}_{\alpha}(n)\big|^{2}.
\end{equation}
For the upper bound, the right-hand side of \eqref{Crsup} gives $\big|\rho^{(\alpha)}_n(h)\big|^{2r}\le(c'_\alpha)^{2r}\min\{1,nt\}^{2r}$ for every $h\in[0,t]$ and every $n$, and the bound is uniform in $h$; taking the supremum over $h\in[0,t]$ in \eqref{pch2} therefore yields $\Omega_r(L;f;t)_2^{2}\le(c'_\alpha)^{2r}\mathcal{S}_r(f;t)^2$. (This is also the case $p=2$ of Theorem \ref{estimatesdual}, for which $c^{*}_{\alpha,2}=1$.)
For the lower bound, replace the supremum by the mean and use Tonelli's theorem and \eqref{Crmean} with $s=2$:
\begin{align*}
\Omega_r(L;f;t)_2^{2}\ &\ge\ \frac1t\int_{0}^{t}\big\|\Delta^{r}_{h}f\big\|^{2}_{\mathcal{L}^2_{\alpha}}dh
=\sum_{n=0}^{\infty}\big|\widehat{f}_{\alpha}(n)\big|^{2}\,\frac1t\int_{0}^{t}\big|\rho^{(\alpha)}_n(h)\big|^{2r}dh\\
&\ge\ c_\alpha^{\,2r}\,\mathcal{S}_r(f;t)^2. \qedhere
\end{align*}
\end{proof}
\begin{corollary}\label{tailform}
Let $\alpha\ge0$, $r\in\mathbb{N}$, $f\in\mathcal{L}^2_{\alpha}(\mathbb{R}^{+})$ and $m\in\mathbb{N}$. Then
\begin{equation}\label{tail}
\mathcal{S}_r\big(f;\tfrac1m\big)^{2}=\sum_{n\ge m}\big|\widehat{f}_{\alpha}(n)\big|^{2}+\frac{1}{m^{2r}}\sum_{n<m}n^{2r}\big|\widehat{f}_{\alpha}(n)\big|^{2},
\end{equation}
and consequently, writing $\mathcal{T}_{r,m}(f)$ for the right-hand side of \eqref{tail},
\[
c_\alpha^{\,2r}\,\mathcal{T}_{r,m}(f)\ \le\ \Omega_r\big(L;f;\tfrac1m\big)_2^{2}\ \le\ (c'_\alpha)^{2r}\,\mathcal{T}_{r,m}(f).
\]
\end{corollary}
\begin{proof}
Identity \eqref{tail} is immediate from \eqref{Sdef}, since $\min\{1,\frac nm\}=\frac nm$ for $n<m$ and $\min\{1,\frac nm\}=1$ for $n\ge m$. The two-sided bound is then Theorem \ref{equiv} with $t=1/m$.
\end{proof}
The next statement shows that the weight $\min\{1,nt\}^{r}$ is optimal up to constant factors in the $L^2$ lower estimate.
\begin{proposition}\label{sharp}
Let $\alpha\ge0$ and $r\in\mathbb{N}$, and let $\phi:\mathbb{N}_0\times(0,1]\to[0,\infty)$ be such that, for some $K<\infty$,
\[
\left(\sum_{n=0}^{\infty}\phi(n,t)^{2}\big|\widehat{f}_{\alpha}(n)\big|^{2}\right)^{1/2}\ \le\ K\,\Omega_r(L;f;t)_2
.
\]
Assume this holds for all $f\in\mathcal L_\alpha^2$ and $0<t\le1$. Then
\[
\phi(n,t)\ \le\ K\,(c'_\alpha)^{r}\min\{1,nt\}^{r}\qquad\text{for all }n\in\mathbb{N}_0,\ 0<t\le1.
\]
Individual Laguerre polynomials witness this optimal order: for $f=L^{(\alpha)}_m$ one has
\[
c_\alpha^{\,r}\min\{1,mt\}^{r}\ \le\ \Omega_r\big(L;L^{(\alpha)}_m;t\big)_2\ \le\ (c'_\alpha)^{r}\min\{1,mt\}^{r}.
\]
\end{proposition}
\begin{proof}
Fix $n$ and $t$ and test the hypothesis on $f=L^{(\alpha)}_n$, for which $\widehat{f}_{\alpha}(k)=\delta_{nk}$ by \eqref{ort}. The left-hand side equals $\phi(n,t)$, while by Theorem \ref{equiv} the right-hand side is at most
\[
K(c'_\alpha)^{r}\mathcal{S}_r(f;t)=K(c'_\alpha)^{r}\min\{1,nt\}^{r}.
\]
The last display is Theorem \ref{equiv} for the same $f$.
\end{proof}
\section{Best approximation and Lipschitz classes}\label{sec5}
We now derive approximation consequences of the weighted $L^2$ estimate. These belong to the Laguerre approximation theory developed by G\"orlich \cite{GM2}, Abilov \cite{Abi,Abi96}, and Abilov, Abilova and Abilov \cite{AAA}. As explained in the introduction, arbitrary-order moduli and the Lipschitz equivalence of Theorem \ref{tm2} already occur in that theory. The proofs below show directly how they follow from Theorem \ref{equiv}, and distinguish the general-$p$ coefficient criteria from best approximation.
Let $\Pi_m:=\operatorname{span}\big\{L^{(\alpha)}_0,\dots,L^{(\alpha)}_{m-1}\big\}$ be the space of Laguerre polynomials of degree less than $m$, and let
\[
E_m(f)_2:=\inf_{P\in\Pi_m}\|f-P\|_{\mathcal{L}^2_{\alpha}},\qquad m\in\mathbb{N},
\]
denote the best approximation of $f\in\mathcal{L}^2_{\alpha}$ by such polynomials. Since $\{L^{(\alpha)}_n\}_{n\ge0}$ is an orthonormal basis of $\mathcal{L}^2_{\alpha}$, the infimum is attained at the partial sum $\sum_{n<m}\widehat f_\alpha(n)L^{(\alpha)}_n$ and
\begin{equation}\label{Em}
E_m(f)_2^{2}=\sum_{n\ge m}\big|\widehat{f}_{\alpha}(n)\big|^{2}.
\end{equation}
Thus Corollary \ref{tailform} may be rewritten as a \emph{realisation} result: for all $m\in\mathbb{N}$,
\begin{equation}\label{realisation}
\Omega_r\big(L;f;\tfrac1m\big)_2\ \asymp\ \left(E_m(f)_2^{2}+\frac{1}{m^{2r}}\sum_{n<m}n^{2r}\big|\widehat{f}_{\alpha}(n)\big|^{2}\right)^{1/2},
\end{equation}
the two terms being the tail of the expansion and a Bernstein-type quantity for its polynomial part.
\begin{theorem}[Jackson and Bernstein inequalities]\label{jb}
Let $\alpha\ge0$, $r\in\mathbb{N}$ and $f\in\mathcal{L}^2_{\alpha}(\mathbb{R}^{+})$.
\begin{itemize}
\item[\rm(i)] \emph{(Jackson)} For every $m\in\mathbb{N}$,
\[
E_m(f)_2\ \le\ c_\alpha^{-r}\,\Omega_r\big(L;f;\tfrac1m\big)_2.
\]
\item[\rm(ii)] \emph{(Bernstein)} For every $P\in\Pi_m$ and $0<t\le1$,
\[
\Omega_r(L;P;t)_2\ \le\ \big(c'_\alpha\,mt\big)^{r}\,\|P\|_{\mathcal{L}^2_{\alpha}}.
\]
\end{itemize}
\end{theorem}
\begin{proof}
(i) By \eqref{Em} and $\min\{1,\frac nm\}=1$ for $n\ge m$ we have $E_m(f)_2\le\mathcal{S}_r(f;\frac1m)$, and the left inequality in \eqref{equivineq} with $t=1/m$ gives the claim.
(ii) If $P\in\Pi_m$ then $\widehat P_\alpha(n)=0$ for $n\ge m$, so in \eqref{Sdef} only the indices $n<m$ contribute, where $\min\{1,nt\}\le nt\le mt$. Hence $\mathcal{S}_r(P;t)\le(mt)^{r}\|P\|_{\mathcal{L}^2_{\alpha}}$, and the right inequality in \eqref{equivineq} gives the claim.
\end{proof}
Let us now define the Lipschitz classes on \(\mathcal{L}^p_{\alpha}(\mathbb{R}^+)\).
\begin{definition}\label{lip}
 Let $\beta>0$, $r\in\mathbb{N}$ and $1\le p\le\infty$. The Lipschitz space \(\text{Lip}(\beta,r;\mathcal{L}^p_{\alpha})\) consists of all \(f\in\mathcal{L}^p_{\alpha}(\mathbb{R}^+)\) such that
\[\Omega_r(L; f;t)_p\leq M t^{\beta},\qquad 0<t\leq1,\]
for some constant $0<M<\infty$. It is normed by
\[
\|f\|_{\text{Lip}(\beta,r;\mathcal{L}^p_{\alpha})} := \|f\|_{\mathcal{L}^p_{\alpha}} + \sup_{0<t\leq1} \left\{t^{-\beta}\Omega_r(L; f;t)_p\right\}.
\]
Here $\Omega_r(L;f;t)_p$ is the modulus of smoothness \eqref{Modulus} built from the translation operator \eqref{tr-op}. The restriction $t\le1$ is essential, see Remark \ref{whysmall}.
\end{definition}
The following known approximation--smoothness equivalence is recovered here from the weighted coefficient estimate; compare \cite{AAA} and the specialisation described in the introduction. Negzaoui \cite{Neg} studies a related Titchmarsh theorem on the Laguerre hypergroup, whose transform and translation differ from the one-variable polynomial expansion used here.
\begin{theorem}\label{tm2}
 Let $\alpha\ge0$, $r\in\mathbb{N}$, $0<\beta<r$ and $f\in\mathcal{L}^2_{\alpha}(\mathbb{R}^+)$. Then the following are equivalent:
\begin{itemize}
\item[\rm(i)] $f\in \text{Lip}(\beta,r;\mathcal{L}^2_{\alpha})$;
\item[\rm(ii)] $\displaystyle\sum_{n\geq m}\big|\widehat{f}_{\alpha}(n)\big|^{2}=O\big(m^{-2\beta}\big)$ as $m\to\infty$; equivalently, $E_m(f)_2=O(m^{-\beta})$.
\end{itemize}
\end{theorem}
\begin{proof}
$(\mathrm{i})\Rightarrow(\mathrm{ii})$. This is Theorem \ref{jb}(i): $E_m(f)_2\le c_\alpha^{-r}\Omega_r(L;f;\frac1m)_2\le c_\alpha^{-r}Mm^{-\beta}$.
$(\mathrm{ii})\Rightarrow(\mathrm{i})$. Put $\gamma_m:=\sum_{n\ge m}|\widehat{f}_{\alpha}(n)|^{2}$. By hypothesis $\gamma_m=O(m^{-2\beta})$, and since $\gamma_m\le\gamma_1\le\|f\|^2_{\mathcal{L}^2_{\alpha}}$ for every $m$, after enlarging the constant we may assume $\gamma_m\le c\,m^{-2\beta}$ for \emph{all} $m\ge1$; and $|\widehat{f}_{\alpha}(n)|^{2}=\gamma_n-\gamma_{n+1}$. By Corollary \ref{tailform} it suffices to estimate the second term of $\mathcal{T}_{r,m}(f)$. Abel summation and the elementary inequality $n^{2r}-(n-1)^{2r}\le 2r\,n^{2r-1}$ give
\begin{align*}
\sum_{n<m}n^{2r}\big|\widehat{f}_{\alpha}(n)\big|^{2}
&=\sum_{n=1}^{m-1}n^{2r}\big(\gamma_n-\gamma_{n+1}\big)
\ \le\ \gamma_1+\sum_{n=2}^{m-1}\big[n^{2r}-(n-1)^{2r}\big]\gamma_n\\
&\le\ \gamma_1+2r\sum_{n=2}^{m-1}n^{2r-1}\gamma_n
\ \le\ \gamma_1+2rc\sum_{n=2}^{m-1}n^{2r-1-2\beta}.
\end{align*}
Since $\beta<r$ we have $2r-1-2\beta>-1$, whence $\sum_{n=2}^{m-1}n^{2r-1-2\beta}\le C_{r,\beta}\,m^{2r-2\beta}$; as also $2r-2\beta>0$ and $\gamma_1\le\|f\|^2_{\mathcal{L}^2_{\alpha}}$, we conclude
\[
\frac{1}{m^{2r}}\sum_{n<m}n^{2r}\big|\widehat{f}_{\alpha}(n)\big|^{2}\ \le\ C_1\,m^{-2\beta}.
\]
Together with $\gamma_m\le c\,m^{-2\beta}$ this gives $\mathcal{T}_{r,m}(f)\le C_2m^{-2\beta}$ and hence, by Corollary \ref{tailform},
\begin{equation}\label{Ombound}
\Omega_r\big(L;f;\tfrac1m\big)_2^{2}\ \le\ (c'_\alpha)^{2r}C_2\,\frac{1}{m^{2\beta}},\qquad m\ge1.
\end{equation}
Finally, let $0<t\le1$. If $t\le\frac12$, put $m:=\lfloor 1/t\rfloor\ge2$; then $t\le\frac1m$ and $m\ge\frac{1}{2t}$, so Property \ref{prop}\,II) and \eqref{Ombound} give $\Omega_r(L;f;t)_2^{2}\le(c'_\alpha)^{2r}C_2(2t)^{2\beta}$. If $\frac12<t\le1$, then by Property \ref{prop}\,II), \eqref{equivineq} and $\mathcal{S}_r(f;1)\le\|f\|_{\mathcal{L}^2_{\alpha}}$,
\[
\Omega_r(L;f;t)_2\le\Omega_r(L;f;1)_2\le(c'_\alpha)^{r}\|f\|_{\mathcal{L}^2_{\alpha}}\le(c'_\alpha)^{r}\|f\|_{\mathcal{L}^2_{\alpha}}(2t)^{\beta}.
\]
In both cases $\Omega_r(L;f;t)_2\le Mt^{\beta}$ with $M$ independent of $t$, i.e.\ $f\in\text{Lip}(\beta,r;\mathcal{L}^2_{\alpha})$.
\end{proof}
\begin{remark}[Saturation]\label{sat}
The modulus $\Omega_r$ saturates at the order $t^{r}$. Indeed, by Theorem \ref{equiv}, for every $f\in\mathcal{L}^2_{\alpha}$, every $n\ge1$ and every $0<t\le\frac1n$,
\[
\Omega_r(L;f;t)_2\ \ge\ c_\alpha^{\,r}\min\{1,nt\}^{r}\big|\widehat{f}_{\alpha}(n)\big|=c_\alpha^{\,r}(nt)^{r}\big|\widehat{f}_{\alpha}(n)\big|.
\]
Hence $\Omega_r(L;f;t)_2=o(t^{r})$ as $t\to0^{+}$ forces $\widehat{f}_{\alpha}(n)=0$ for every $n\ge1$, that is, $f$ is a multiple of $L^{(\alpha)}_0$. In particular $\text{Lip}(\beta,r;\mathcal{L}^2_{\alpha})$ consists of the constants alone when $\beta>r$. For $r=1$ Theorem \ref{tm2} recovers the classical range $0<\beta<1$.
\end{remark}
\begin{remark}[The restriction $\beta<r$ is sharp]\label{betaeqr}
For $f\in\mathcal L_\alpha^2$, the endpoint has the exact coefficient characterisation
\begin{equation}\label{endpointclass}
\Omega_r(L;f;t)_2=O(t^r)\quad(t\downarrow0)
\quad\Longleftrightarrow\quad
\sum_{n\ge1}n^{2r}|\widehat f_\alpha(n)|^2<\infty.
\end{equation}
Indeed, the lower bound in Theorem \ref{equiv} bounds the partial sums with $n\le1/t$ after division by $t^{2r}$; letting $t\downarrow0$ proves necessity. Conversely, $\min\{1,nt\}\le nt$ and the upper bound in that theorem prove sufficiency.

The approximation-rate condition in Theorem \ref{tm2} alone does not characterise this endpoint class. Let $\widehat{f}_{\alpha}(0):=0$ and $\widehat{f}_{\alpha}(n):=n^{-r-\frac12}$ for $n\ge1$; then $f\in\mathcal{L}^2_{\alpha}$ and, by \eqref{Em},
\[
E_m(f)_2^{2}=\sum_{n\ge m}n^{-2r-1}\asymp m^{-2r},
\]
so that condition {\rm(ii)} of Theorem \ref{tm2} holds with $\beta=r$. On the other hand, splitting \eqref{Sdef} at $n=1/t$,
\begin{align*}
\mathcal S_r(f;t)^2
&=t^{2r}\sum_{1\le n\le1/t}n^{-1}+\sum_{n>1/t}n^{-2r-1}\\
&\asymp t^{2r}\log\tfrac1t+t^{2r}
\asymp t^{2r}\log\tfrac1t,\qquad t\downarrow0.
\end{align*}
so Theorem \ref{equiv} gives $\Omega_r(L;f;t)_2\asymp t^{r}\big(\log\frac1t\big)^{1/2}$, which is not $O(t^{r})$. Thus {\rm(ii)} does not imply {\rm(i)} when $\beta=r$. Note also that $\text{Lip}(r,r;\mathcal{L}^2_{\alpha})$ is strictly larger than the constants --- replacing the coefficients above by $n^{-r-\frac12}(\log n)^{-1}$, $n\ge2$, makes the last sum convergent and yields a non-constant $f$ with $\Omega_r(L;f;t)_2\asymp t^{r}$ --- so the endpoint class is non-trivial but is not characterised by the decay of $E_m(f)_2$ alone.
\end{remark}
\begin{example}\label{ex}
Fix $r\in\mathbb{N}$ and $0<\beta<r$, and let $f\in\mathcal{L}^2_{\alpha}(\mathbb{R}^+)$ be defined by its Laguerre coefficients
\[
\widehat{f}_{\alpha}(0):=0,\qquad \widehat{f}_{\alpha}(n):=n^{-\beta-\frac12}\ \ (n\ge1),
\]
which lies in $\mathcal{L}^2_{\alpha}$ by \eqref{pe} because $\sum_{n\ge1}n^{-2\beta-1}<\infty$. Splitting \eqref{Sdef} at $n=1/t$ and using $2r-2\beta-1>-1$,
\[
\mathcal{S}_r(f;t)^{2}=t^{2r}\!\!\sum_{1\le n\le 1/t}\!\! n^{2r-2\beta-1}+\sum_{n>1/t}n^{-2\beta-1}\ \asymp\ t^{2r}\,t^{-(2r-2\beta)}+t^{2\beta}\ \asymp\ t^{2\beta},
\]
so Theorem \ref{equiv} gives $\Omega_r(L;f;t)_2\asymp t^{\beta}$ for $0<t\le1$. Hence $f\in\text{Lip}(\beta,r;\mathcal{L}^2_{\alpha})$, while $f\notin\text{Lip}(\beta',r;\mathcal{L}^2_{\alpha})$ for \emph{every} $\beta'>\beta$. Consequently the classes $\text{Lip}(\beta,r;\mathcal{L}^2_{\alpha})$, $0<\beta\le r$, are pairwise distinct, and the exponent in Theorem \ref{tm2} is exact.
\end{example}
We close by extending Theorem \ref{tm2} to the remaining exponents. For $p\ne2$ the role of the best approximation $E_m(f)_2$ is played by the tail of the Laguerre coefficients measured in ${l}^{p'}_{\alpha}$:
\begin{equation}\label{Emp}
\mathcal{E}_m(f)_p:=\left(\sum_{n\ge m}(\lambda^{\alpha}_n)^{\frac{2-p'}{2}}\big|\widehat{f}_{\alpha}(n)\big|^{p'}\right)^{1/p'},\qquad m\in\mathbb{N},
\end{equation}
with the usual modification $\mathcal{E}_m(f)_1:=\sup_{n\ge m}(\lambda^{\alpha}_n)^{-1/2}\big|\widehat{f}_{\alpha}(n)\big|$ for $p=1$. By \eqref{Em} we have $\mathcal{E}_m(f)_2=E_m(f)_2$. We also write $E_m(f)_p:=\inf_{P\in\Pi_m}\|f-P\|_{\mathcal{L}^p_{\alpha}}$, consistently with the case $p=2$ above.
\begin{proposition}\label{jbp}
Let $\alpha\ge0$, $2\le p\le\infty$ and $f\in\mathcal{L}^p_{\alpha}(\mathbb{R}^+)$. Then
\[
E_m(f)_p\ \le\ c^{*}_{\alpha,p}\,\mathcal{E}_m(f)_p,\qquad m\in\mathbb{N}.
\]
\end{proposition}
\begin{proof}
We may assume $\mathcal{E}_m(f)_p<\infty$. Put $S_mf:=\sum_{n<m}\widehat{f}_{\alpha}(n)L^{(\alpha)}_n\in\Pi_m$. By Theorem \ref{hyldual}\,{\rm(i)} applied to the sequence which equals $\widehat{f}_{\alpha}(n)$ for $n\ge m$ and $0$ otherwise, the corresponding series converges in $\mathcal{L}^p_{\alpha}$ to a function $g$ with $\|g\|_{\mathcal{L}^p_{\alpha}}\le c^{*}_{\alpha,p}\mathcal{E}_m(f)_p$ and with those coefficients. The function $f-S_mf$ lies in $\mathcal{L}^p_{\alpha}$ and has the same coefficients, so the uniqueness argument in the proof of Theorem \ref{hyldual}\,{\rm(ii)} gives $f-S_mf=g$ almost everywhere. Hence $E_m(f)_p\le\|f-S_mf\|_{\mathcal{L}^p_{\alpha}}=\|g\|_{\mathcal{L}^p_{\alpha}}\le c^{*}_{\alpha,p}\mathcal{E}_m(f)_p$.
\end{proof}
\begin{theorem}\label{tm2p}
Let $\alpha\ge0$, $r\in\mathbb{N}$, $1\le p\le\infty$ and $f\in\mathcal{L}^p_{\alpha}(\mathbb{R}^+)$.
\begin{itemize}
\item[\rm(i)] Let $1\le p\le2$ and $\beta>0$. If $f\in\mathrm{Lip}(\beta,r;\mathcal{L}^p_{\alpha})$, then $\mathcal{E}_m(f)_p=O\big(m^{-\beta}\big)$ as $m\to\infty$.
\item[\rm(ii)] Let $2\le p\le\infty$ and $0<\beta<r$. If $\mathcal{E}_m(f)_p=O\big(m^{-\beta}\big)$, then $f\in\mathrm{Lip}(\beta,r;\mathcal{L}^p_{\alpha})$; moreover $E_m(f)_p=O(m^{-\beta})$.
\end{itemize}
For $p=2$ the two statements are converse to one another and, by \eqref{Em}, reduce to Theorem \ref{tm2}.
\end{theorem}
\begin{proof}
{\rm(i)} Apply Theorem \ref{estimates} with $t=\frac1m$ and discard the terms with $n<m$. For $n\ge m$ one has $\min\{1,\frac nm\}=1$, so the left-hand side of \eqref{main1} is at least $\mathcal{E}_m(f)_p$ when $1<p\le2$, and the left-hand side of \eqref{main2} is at least $\mathcal{E}_m(f)_1$ when $p=1$. In either case
$\mathcal{E}_m(f)_p\le c\,\Omega_r\big(L;f;\tfrac1m\big)_p\le c\,M\,m^{-\beta}$.
{\rm(ii)} Write $w_n:=(\lambda^{\alpha}_n)^{\frac{2-p'}{2}}$ and $\gamma_m:=\mathcal{E}_m(f)_p^{p'}=\sum_{n\ge m}w_n|\widehat{f}_{\alpha}(n)|^{p'}$, so that $w_n|\widehat{f}_{\alpha}(n)|^{p'}=\gamma_n-\gamma_{n+1}$; note $1\le p'\le2$. Since $\gamma_m\le\gamma_1<\infty$ for every $m$, after enlarging the constant we may assume $\gamma_m\le c\,m^{-\beta p'}$ for all $m\ge1$. Because $\min\{1,\frac nm\}=\frac nm$ for $n<m$ and $=1$ for $n\ge m$, Theorem \ref{estimatesdual} with $t=\frac1m$ gives
\[
\Omega_r\big(L;f;\tfrac1m\big)_p^{p'}\ \le\ \big(c^{*}_{\alpha,p}\big)^{p'}(c'_\alpha)^{rp'}\left[\gamma_m+\frac{1}{m^{rp'}}\sum_{n<m}n^{rp'}w_n\big|\widehat{f}_{\alpha}(n)\big|^{p'}\right].
\]
Abel summation together with $n^{rp'}-(n-1)^{rp'}\le rp'\,n^{rp'-1}$ yields
\[
\sum_{n<m}n^{rp'}w_n\big|\widehat{f}_{\alpha}(n)\big|^{p'}\ \le\ \gamma_1+rp'\sum_{n=2}^{m-1}n^{rp'-1}\gamma_n\ \le\ \gamma_1+rp'c\sum_{n=2}^{m-1}n^{rp'-1-\beta p'} ,
\]
and $\beta<r$ gives $rp'-1-\beta p'>-1$, so the last sum is at most $C_{r,\beta,p}\,m^{rp'-\beta p'}$. Since also $\beta<r$ gives $m^{-rp'}\gamma_1\le\gamma_1m^{-\beta p'}$, we conclude that $\Omega_r(L;f;\frac1m)_p\le C_1m^{-\beta}$ for all $m\ge1$. Finally, if $0<t\le\frac12$ put $m:=\lfloor 1/t\rfloor\ge2$, so that $t\le\frac1m$ and $m\ge\frac{1}{2t}$; Property \ref{prop}\,II) then gives $\Omega_r(L;f;t)_p\le C_1(2t)^{\beta}$, while for $\frac12<t\le1$ Property \ref{prop}\,II) and the bound $\Omega_r(L;f;1)_p\le(1+B_{p,\alpha})^{r}\|f\|_{\mathcal{L}^p_{\alpha}}$ of Section \ref{sec2} give $\Omega_r(L;f;t)_p\le(1+B_{p,\alpha})^{r}\|f\|_{\mathcal{L}^p_{\alpha}}(2t)^{\beta}$. Hence $f\in\mathrm{Lip}(\beta,r;\mathcal{L}^p_{\alpha})$. The last assertion is Proposition \ref{jbp}.
\end{proof}
\begin{remark}
For $p\ne2$ the two implications of Theorem \ref{tm2p} point in opposite directions, and the results above do not decide whether they combine into an equivalence. What Remark \ref{whyp2} rules out is the reversal, for a fixed $t$, of the estimate \eqref{main1}; it says nothing about the two asymptotic conditions of Theorem \ref{tm2p}, and indeed its test functions $f=L^{(\alpha)}_m$ satisfy both of them simultaneously. We do not pursue the question here.
\end{remark}
\section*{Statements and Declarations}
\subsection*{Funding}
This research was funded by the Committee of Science of the Ministry of Science and Higher Education of the Republic of Kazakhstan (Grant No. BR31714735).
\subsection*{Data availability}
No datasets were generated or analysed during the present study.

\end{document}